\documentclass[11pt]{article}

\usepackage{amsmath,amssymb,amsfonts,textcomp,amsthm,xifthen,psfrag,graphicx,color,MnSymbol}
\usepackage[T1]{fontenc}

\usepackage{hyperref}

\date{}

\newtheorem{theorem}{Theorem}
\newtheorem{lemma}[theorem]{Lemma}

\newtheorem{prop}[theorem]{Proposition}
\newtheorem{remark}[theorem]{Remark}
\theoremstyle{definition} 

\DeclareMathOperator*{\argmin}{arg\, min}

\newcommand{\<}{\langle{}}
\renewcommand{\>}{\rangle}

\newcommand{\ip}[2]{\llangle#1\hspace*{.5mm},#2\rrangle}
\newcommand{\dual}[2]{\<#1\hspace*{.5mm},#2\>}
\newcommand{\vdual}[2]{(#1\hspace*{.5mm},#2)}

\newcommand{\norm}[2]{\|#1\|_{#2}}

\newcommand{\diam}{\mathrm{diam}}
\newcommand{\wilde}{\widetilde}
\newcommand{\wat}{\widehat}
\newcommand{\transp}{\mathsf{T}}

\def\Grad{\boldsymbol{\varepsilon}}
\def\pwGrad{\boldsymbol{\varepsilon}_\cT}
\def\Div{{\rm\bf div\,}}
\def\pwDiv{{\rm\bf div}_\cT}
\def\grad{\nabla}
\def\pwgrad{\nabla_\cT}
\def\MM{\mathbf{M}}
\def\QQ{\mathbf{Q}}

\def\TTheta{\mathbf{\Theta}}

\def\bq{\boldsymbol{q}}

\newcommand{\bL}{\ensuremath{\mathbf{L}}}
\newcommand{\LL}{\ensuremath{\mathbb{L}}}
\newcommand{\HH}{\ensuremath{\mathbf{H}}}
\newcommand{\PP}{\ensuremath{\mathbb{P}}}

\def\btheta{\boldsymbol{\theta}}
\def\tQ{\wat{\boldsymbol{q}}}
\def\tu{\wat{\boldsymbol{u}}}

\def\tz{\wat{\boldsymbol{z}}}

\def\tv{\wat{\boldsymbol{v}}}
\newcommand{\EE}{\ensuremath{\mathbb{E}}}
\newcommand{\UU}{\ensuremath{\mathcal{U}}}
\newcommand{\VV}{\ensuremath{\mathcal{V}}}

\newcommand{\bP}{\ensuremath{\mathbf{P}}}

\newcommand{\br}{\ensuremath{\mathbf{r}}}

\newcommand{\bj}{\ensuremath{\mathbf{j}}}
\newcommand{\bl}{\ensuremath{\boldsymbol{\ell}}}
\newcommand{\bu}{\ensuremath{\mathbf{u}}}

\newcommand{\vv}{\ensuremath{\mathbf{v}}}
\newcommand{\uu}{\ensuremath{\mathbf{u}}}
\newcommand{\ww}{\ensuremath{\mathbf{w}}}

\newcommand{\bw}{\ensuremath{\mathbf{w}}}

\newcommand{\XXi}{\ensuremath{\boldsymbol{\Xi}}}
\newcommand{\deltaXXi}{\ensuremath{\boldsymbol{\delta\!\Xi}}}

\newcommand{\deltaz}{\delta\!z}

\newcommand{\bdeltaq}{\boldsymbol{\delta}\!\bq}
\newcommand{\bdeltau}{\boldsymbol{\delta}\!\bu}

\newcommand{\traceDD}[1]{\mathrm{tr}_{#1}^{\mathrm{dDiv}}} 
\newcommand{\traceGG}[1]{\mathrm{tr}_{#1}^{\mathrm{Ggrad}}} 
\newcommand{\ttraceGG}[1]{{\wilde{\rm tr}_{#1}^{\mathrm{Ggrad}}}} 
\newcommand{\bH}{\ensuremath{\mathbf{H}}}

\newcommand{\trggrad}[1]{{\mathrm{Ggrad},#1}}
\newcommand{\trddiv}[1]{{\mathrm{dDiv},#1}}

\def\div{{\rm div\,}}

\def\pwdiv{ {\rm div}_\cT\,}

\newcommand{\jump}[1]{[#1]}
\newcommand{\jjump}[1]{\lsem #1\rsem}

\newcommand{\HdDiv}[1]{{H(\div\Div\!,#1)}}
\newcommand{\HDivdiv}[1]{{\bH(\Div\!,\div\!,#1)}}
\def\Divdiv{{\Div\!,\div\!}}
\def\dDiv{{\div\Div\!}}

\def\divref{\widehat\div}
\def\Divref{\widehat\Div}
\def\Gradref{\widehat\Grad}
\def\gradref{\widehat\grad}

\newcommand{\HdDivref}[1]{{H(\divref\Divref\!,#1)}}
\newcommand{\HDivdivref}[1]{{\bH(\Divref\!,\divref\!,#1)}}
\def\Divdivref{{\Divref\!,\divref\!}}
\def\dDivref{{\divref\Divref\!}}

\newcommand{\ttt}{{\rm T}}

\def\projGG{\Pi^{\mathrm{Ggrad}}} 
\def\projDD{\Pi^{\mathrm{dDiv}}} 
\def\projD{\Pi^{\mathrm{Div,div}}} 
\def\projDQ{\Pi^{\mathrm{Div}}}
\def\projDtau{\Pi^{\mathrm{div}}}

\newcommand{\di}{d}
\newcommand{\R}{\ensuremath{\mathbb{R}}}
\newcommand{\N}{\ensuremath{\mathbb{N}}}

\newcommand{\nn}{\ensuremath{\mathbf{n}}}

\newcommand{\cC}{\ensuremath{\mathcal{C}}}
\newcommand{\cCinv}{\ensuremath{\mathcal{C}^{-1}}}
\newcommand{\cT}{\ensuremath{\mathcal{T}}}

\newcommand{\cS}{\ensuremath{\mathcal{S}}}

\newcommand{\bt}{\ensuremath{\mathbf{t}}}

\newcommand{\OO}{\ensuremath{\mathcal{O}}}
\newcommand{\cE}{\ensuremath{\mathcal{E}}}
\newcommand{\cN}{\ensuremath{\mathcal{N}}}

\newcommand{\btau}{{\boldsymbol\tau}}
\newcommand{\deltabtau}{{\boldsymbol{\delta\!\tau}}}

\title{Fully discrete DPG methods for the Kirchhoff--Love plate bending model
\thanks{Supported by CONICYT through FONDECYT projects 1150056 and 11170050}}
\author{
Thomas~F\"uhrer$^\dagger$
\and
Norbert Heuer\thanks{
Facultad de Matem\'aticas, Pontificia Universidad Cat\'olica de Chile,
Avenida Vicu\~na Mackenna 4860, Santiago, Chile,
email: {\tt \{tofuhrer,nheuer\}@mat.uc.cl}}
}

\begin{document}
\maketitle
\begin{abstract}
We extend the analysis and discretization of the Kirchhoff--Love plate bending problem from
[T.~F\"uhrer, N.~Heuer, A.H.~Niemi,
\emph{An ultraweak formulation of the Kirchhoff--Love plate bending model and DPG approximation},
arXiv:~1805.07835, 2018] in two aspects. First, we present a well-posed formulation and quasi-optimal
DPG discretization that includes the gradient of the deflection.
Second, we construct Fortin operators that prove the well-posedness and quasi-optimal convergence
of lowest-order discrete schemes with approximated test functions for both formulations.
Our results apply to the case of non-convex polygonal plates where shear forces can be
less than $L_2$-regular. Numerical results illustrate expected convergence orders.


\end{abstract}

\section{Introduction}

In \cite{FuehrerHN_UFK} we presented an ultraweak formulation of the Kirchhoff--Love
plate bending problem and proposed a stable and quasi-optimally converging discontinuous
Petrov--Galerkin scheme with optimal test functions (DPG method).
Main contributions of that work are the setup and analysis of certain product spaces
with corresponding traces and jumps, and the design of a formulation that does not
require unreasonable regularity of the solution.
Specifically, the shear force vector is not required to be $L_2$-regular
so that the bending moment tensor $\MM$ is not necessarily an element of $\bH(\Div\!,\Omega)$
(the space of symmetric $L_2(\Omega)$-tensors with vector-valued divergence in $L_2(\Omega)$).
It is known that the appropriate space for $\MM$ is $\HdDiv{\Omega}$ of symmetric $L_2(\Omega)$-tensors
with twice iterated divergence in $L_2(\Omega)$, cf.~Amara \emph{at al.} \cite{AmaraCPC_02_BMM},
see also Rafetseder and Zulehner \cite{RafetsederZ_DRK}. This space has been at the center
of our attention in \cite{FuehrerHN_UFK}, and for a more detailed discussion of the
Kirchhoff--Love model and related references, we refer to that reference.

In this paper we continue to study and extend ultraweak formulations of the Kirchhoff--Love
model and DPG discretizations. Specifically, we extend the formulation from
\cite{FuehrerHN_UFK} to one that includes the gradient of the deflection, $\grad u$, as an
independent unknown $\btheta$. Whereas this variable is of interest by itself, we consider this study
also as a preparation to deal with the Reissner--Mindlin model where the related rotation
vector is of importance. In fact, our analysis illustrates that the principal trace and jump
operators from \cite{FuehrerHN_UFK} cover the new case with gradient variable $\btheta$.
We expect that the Reissner--Mindlin model can be covered by an extension of
the Kirchhoff--Love model.

A second contribution of this work relates to the test functions of the DPG method.
In its original version, the DPG method uses \emph{optimal} test functions, see \cite{DemkowiczG_11_CDP}.
These test functions are selected within the test space of the variational formulation
and, thus, their calculation requires the solution of problems in this infinite-dimensional
space. In practice, the infinite-dimensional test space is replaced by a finite-dimensional
subspace of dimension larger than that of the discrete approximation space.
Then, the calculation of test functions becomes feasible.
In \cite{GopalakrishnanQ_14_APD}, Gopalakrishnan and Qiu analyzed this setting
and called it \emph{practical} DPG method in contrast to the original \emph{ideal} method.
In this paper we refer to the practical DPG scheme as the fully discrete one.
As the optimal test functions of the ideal DPG scheme guarantee a uniform discrete
inf-sup property, it is no surprise that the analysis in \cite{GopalakrishnanQ_14_APD} is
based on the construction of Fortin operators.
In this work we do precisely this. We construct Fortin operators for the inherent spaces
of our two ultraweak formulations, with and without gradient variable.
In this way, we prove the well-posedness and quasi-optimal convergence of both fully discrete
schemes, with and without gradient variable. One of the essential ingredients is
to use appropriate transformations of vector and tensor-valued functions.
Whereas it is well known that vector-valued functions are isomorphically transformed
between corresponding spaces by the Piola transformation, it does not maintain symmetry
of tensor-valued functions. We therefore employ a symmetrized Piola transformation
which is known from the second Piola--Kirchhoff stress tensor. We refer to it as the
Piola--Kirchhoff transformation. In \cite{PechsteinS_11_TDN}, Pechstein and Sch\"oberl
employed this very transformation, though considering the Hellinger--Reissner model of elasticity
where the twice iterated divergence of the stress is an element of $H^{-1}(\Omega)$, the dual
of $H^1(\Omega)$ with homogeneous Dirichlet condition on a part of the boundary.

The remainder of this work is as follows. In the next section we briefly present our
form of the Kirchhoff--Love plate bending model, recall Sobolev spaces and traces
in \S\ref{sec_traces}, and present an ultraweak variational formulation in \S\ref{sec_VF}.
Furthermore, \S\S\ref{sec_DPG} and~\ref{sec_DPG_discrete} study well-posedness of the
continuous and discrete formulations, in the former subsection when using optimal test
functions and in the latter for the fully discrete scheme. The fully discrete
case from \cite{FuehrerHN_UFK} without gradient variable is discussed in \S\ref{sec_DPG_discrete2}.
Proofs of the well-posedness of the ultraweak formulation (Theorem~\ref{thm_stab})
and the quasi-optimal convergence with optimal test functions (Theorem~\ref{thm_DPG})
are given in Section~\ref{sec_adj}.
As is standard in DPG analysis, main part is to identify and analyze the adjoint problem.
Section~\ref{sec_Fortin} presents all the required Fortin operators.
Vector and tensor transformations are presented in \S\ref{sec_trafo}, and some required
basis functions and discrete spaces are studied in \S\ref{sec_basis}.
Fortin operators for three different spaces are designed in the remaining three subsections.
All operators are initially defined on a reference element. As illustrated by
Nagaraj \emph{et al.} \cite{NagarajPD_17_CDF}, the local construction can be done
by defining a mixed problem with automatically invertible diagonal block and an
off-diagonal block whose injectivity has to be checked. We verify the injectivity of
this block in one case directly and in the other two cases by MATLAB calculations which are not reported here.
Finally, in Section~\ref{sec_num} we present some numerical experiments that underline
the expected convergence of our fully discrete scheme. In particular, the chosen example
is such that the shear force is not an $L_2$ vector field.

Throughout, the notation $a\lesssim b$ means that there exists a constant $c>0$
that is independent of the underlying mesh such that $a\le cb$.

\section{Model problem and DPG method} \label{sec_model}

Let $\Omega\subset\R^d$ ($d=2,3$) be a bounded simply connected Lipschitz domain with boundary
$\Gamma=\partial\Omega$ and exterior unit normal vector $\nn$. Our model problem is
\begin{subequations} \label{prob}
\begin{alignat}{2}
    -\div\Div\MM             &= f  && \quad\text{in} \quad \Omega\label{p1},\\
    \cCinv\MM + \Grad\btheta &= 0  && \quad\text{in} \quad \Omega\label{p3},\\
    \btheta - \grad u  &= 0  && \quad\text{in} \quad \Omega\label{p4},\\
    u = 0,\quad \nn\cdot\grad u &= 0 &&\quad\text{on}\quad\Gamma.\label{pBC}
\end{alignat}
\end{subequations}
Here, $\btheta$ and $\MM$ are vector-valued and symmetic tensor functions, respectively,
and $\cC$ is a constant fourth-order tensor that induces a self-adjoint isomorphism
from $\LL_2^s(\Omega)$ to $\LL_2^s(\Omega)$, the space of
symmetric second-order tensors with $L_2(\Omega)$ components.
The operator $\Div$ is the divergence operator applied row-wise to tensors and
$\Grad\btheta:=1/2(\grad\btheta+\grad\btheta^\transp)$ so that $\Grad\btheta=\Grad\grad u$ is the
Hessian matrix of $u$.

For $d=2$, \eqref{prob} is a simplified and re-scaled version of the Kirchhoff--Love model,
cf.~\cite{FuehrerHN_UFK} for details.
For its relevance for other fourth-order problems we present our continuous formulation and analysis
also for three space dimensions. The fully discrete analysis will only be provided for two
space dimensions.

In \cite{FuehrerHN_UFK} we considered an ultraweak formulation of \eqref{prob}
and DPG approximation without the variable $\btheta$, that is, of the problem
\[
    -\div\Div\MM = f,\quad
    \cCinv\MM + \Grad\grad u = 0\quad\text{in}\ \Omega;\qquad
    u = 0,\ \nn\cdot\grad u = 0\quad\text{on}\ \Gamma.
\]
In this paper we provide an extension that includes a direct approximation of the gradient $\btheta$,
and analyze fully discrete approximations for formulations with and without the variable $\btheta$.
In order to derive an appropriate variational formulation we need to recall some Sobolev
spaces and associated traces, and introduce a new space and an extension of one of the trace operators.

\subsection{Sobolev spaces and traces} \label{sec_traces}

DPG formulations are related with decompositions of the domain into elements.
Specifically, we consider a (family of) mesh(es) $\cT$ that consists of general non-intersecting
Lipschitz elements. In the following, $\cS$ refers to the skeleton of $\cT$ and consists of the
collection of element boundaries, $\cS=\{\partial T;\;T\in\cT\}$.
Later, for the discretization, we restrict ourselves to two space dimensions and consider
shape-regular triangulations $\cT$.

We use the standard spaces $L_2(T)$, $H^1(T)$ and $H^2(T)$ for $T\in\cT$ with analogous notation
for $\Omega$ instead of $T$. The respective completions of spaces of smooth functions
with compact support are $H^1_0(T)$ and $H^2_0(T)$, and correspondingly for $\Omega$. In this form, spaces
denote those of scalar functions. Vector-valued function spaces are indicated by boldface
symbols, e.g., $\bH^1(T)$, and tensor spaces by blackboard symbols, e.g., $\LL_2(\Omega)$. Furthermore,
$\LL_2^s(T)$ and $\LL_2^s(\Omega)$ denote spaces of symmetric tensor functions.
Throughout, $L_2$-norms are denoted by $\|\cdot\|_T$ (for an element $T\in\cT$) and $\|\cdot\|$ for
$L_2$-functions on $\Omega$, for scalar, vector, and tensor spaces.
The $L_2$-bilinear forms are $\vdual{\cdot}{\cdot}_T$ and $\vdual{\cdot}{\cdot}$ for $T\in\cT$ and
$\Omega$, respectively.
The spaces $H^1(T)$ and $H^2(T)$ are provided with the norms
$\|v\|_{1,T}:=\bigl(\|v\|_T^2+\|\grad v\|_T^2\bigr)^{1/2}$ and
$\|v\|_{2,T}:=\bigl(\|v\|_T^2+\|\Grad\grad v\|_T^2\bigr)^{1/2}$, respectively.
We also use the notation $\|\cdot\|_2:=\|\cdot\|_{2,\Omega}$ for the corresponding norm in $H^2(\Omega)$.
The space $\bH^1_0(T)$ consists of vector-valued functions with vanishing trace on $\partial T$,
with norm $|\cdot|_{1,T}:=\|\Grad\cdot\|_T$.
There are corresponding product spaces related to the mesh $\cT$ with canonical product norms,
e.g., $H^1(\cT)$ and $\bH^1_0(\cT)$ with norms $\|\cdot\|_{1,\cT}$ and $|\cdot|_{1,\cT}$,
respectively, and similarly for the other spaces.
$L_2(\cT)$-dualities are indicated by $\vdual{\cdot}{\cdot}_\cT$, i.e.,
appearing differential operators are taken piecewise with respect to $\cT$. The corresponding
$L_2(\Omega)$-norm is $\|\cdot\|_\cT$.

Furthermore, for $T\in\cT$, we need the space
\(
    \HdDiv{T} 
\)
as the completion of smooth symmetric tensors with respect to the norm
\[
   \|\TTheta\|_{\dDiv,T}:=\bigl(\|\TTheta\|_T^2+\|\div\Div\TTheta\|_T^2\bigr)^{1/2}.
\]
The induced product space is denoted by $\HdDiv{\cT}$ with norm $\|\cdot\|_{\dDiv,\cT}$.
Analogously, we define the global space
\(
    \HdDiv{\Omega} 
\)
with norm
\[
   \|\TTheta\|_{\div\Div}:=\bigl(\|\TTheta\|^2+\|\div\Div\TTheta\|^2\bigr)^{1/2}.
\]
In \cite{FuehrerHN_UFK}, we studied two trace operators that stem from twice integrating by parts
the term $\Grad\grad u$,
\begin{equation} \label{trGGT}
   \traceGG{T}:\;
   \left\{\begin{array}{cll}
      H^2(T) & \to & \HdDiv{T}',\\
      z & \mapsto & \dual{\traceGG{T}(z)}{\TTheta}_{\partial T} :=
                    \vdual{\div\Div\TTheta}{z}_T - \vdual{\TTheta}{\Grad\grad z}_T
   \end{array}\right.
\end{equation}
and
\begin{equation} \label{trDDT}
   \traceDD{T}:\;
   \left\{\begin{array}{cll}
      \HdDiv{T} & \to & H^2(T)',\\
      \TTheta & \mapsto & \dual{\traceDD{T}(\TTheta)}{z}_{\partial T} :=
                    \vdual{\div\Div\TTheta}{z}_T - \vdual{\TTheta}{\Grad\grad z}_T
   \end{array}\right.
\end{equation}
for $T\in\cT$, and the corresponding collective trace operators
\begin{equation} \label{trGG}
   \traceGG{}:\;
   \left\{\begin{array}{cll}
      H^2_0(\Omega) & \to & \HdDiv{\cT}',\\
      z & \mapsto & \traceGG{}(z) := (\traceGG{T}(z))_T
   \end{array}\right.
\end{equation}
and
\[
   \traceDD{}:\;
   \left\{\begin{array}{cll}
      \HdDiv{\Omega} & \to & H^2(\cT)',\\
      \TTheta & \mapsto & \traceDD{}(\TTheta) := (\traceGG{T}(\TTheta))_T
   \end{array}\right.
\]
with dualities (for $z\in H^2_0(\Omega)$, $\TTheta\in \HdDiv{\cT}$)
\begin{equation} \label{ipGG}
   \dual{\traceGG{}(z)}{\TTheta}_\cS
   := \sum_{T\in\cT} \dual{\traceGG{T}(z)}{\TTheta}_{\partial T}
   = \vdual{\div\Div\TTheta}{z}_\cT - \vdual{\TTheta}{\Grad\grad z}
\end{equation}
and (for $\TTheta\in\HdDiv{\Omega}$, $z\in H^2(\cT)$)
\begin{equation} \label{ipDD}
   \dual{\traceDD{}(\TTheta)}{z}_\cS
   := \sum_{T\in\cT} \dual{\traceDD{T}(\TTheta)}{z}_{\partial T}
   = \vdual{\div\Div\TTheta}{z} - \vdual{\TTheta}{\Grad\grad z}_\cT.
\end{equation}
These trace operators give rise to the trace spaces
\begin{align*}
   \bH^{3/2,1/2}(\partial T) &:= \traceGG{T}(H^2(T))\quad (T\in\cT),\qquad
   \bH^{3/2,1/2}_{00}(\cS)    := \traceGG{}(H^2_0(\Omega))
\end{align*}
and
\begin{align*}
   \bH^{-3/2,-1/2}(\partial T) &:= \traceDD{T}(\HdDiv{T})\quad (T\in\cT),\\
   \bH^{-3/2,-1/2}(\cS)        &:= \traceDD{}(\HdDiv{\Omega}).
\end{align*}
Of course, the duality relations \eqref{ipGG}, \eqref{ipDD} amount to integration-by-parts formulas,
and the implied dualities for $\tz\in\bH^{3/2,1/2}_{00}(\cS)$ and $\tQ\in\bH^{-3/2,-1/2}(\cS)$ are
\begin{align} \label{dualityGG}
   \dual{\tz}{\TTheta}_\cS &:= \dual{\traceGG{}(z)}{\TTheta}_\cS
   \quad\forall\TTheta\in\HdDiv{\cT}
\end{align}
with $z\in H^2_0(\Omega)$ such that $\traceGG{}(z)=\tz$, and
\begin{align} \label{dualityDD}
   \dual{\tQ}{z}_\cS &:= \dual{\traceDD{}(\TTheta)}{z}_\cS
   \quad\forall z\in H^2(\cT)
\end{align}
with $\TTheta\in\HdDiv{\Omega}$ such that $\traceDD{}(\TTheta)=\tQ$.
The trace spaces are provided with the norms
\begin{alignat*}{2}
   \|\tz\|_{3/2,1/2,\partial T} &:= 
   \sup_{0\not=\TTheta\in\HdDiv{T}} \frac{\dual{\tz}{\TTheta}_{\partial T}}{\|\TTheta\|_{\dDiv,T}}
   \qquad&&
   \begin{cases}\text{using duality \eqref{trGGT} with } z\in H^2(T)\\
                \text{such that } \traceGG{T}(z)=\tz \quad(T\in\cT),
   \end{cases}\\
   \|\tz\|_{3/2,1/2,00,\cS} &:=
   \sup_{0\not=\TTheta\in\HdDiv{\cT}} \frac{\dual{\tz}{\TTheta}_\cS}{\|\TTheta\|_{\dDiv,\cT}}
   &&\hspace{0.8em}\text{using duality } \eqref{dualityGG},
\end{alignat*}
and
\begin{alignat}{2}
   \|\tQ\|_{-3/2,-1/2,\partial T} &:= \sup_{0\not=z\in H^2(T)} \frac{\dual{\tQ}{z}_{\partial T}}{\|z\|_{2,T}}
   \qquad&&
   \begin{cases}\text{using duality \eqref{trDDT} with } \TTheta\in\HdDiv{T}\\
                \text{such that } \traceDD{T}(\TTheta)=\tQ\quad(T\in\cT),
   \end{cases}\nonumber\\
   \|\tQ\|_{-3/2,-1/2,\cS} &:=  \sup_{0\not=z\in H^2(\cT)} \frac{\dual{\tQ}{z}_\cS}{\|z\|_{2,\cT}}
   &&\hspace{0.8em}\text{using duality } \eqref{dualityDD},
   \label{def_norm_trDD}
\end{alignat}
respectively. By \cite[Lemmas 2, 3 and Propositions 5, 9]{FuehrerHN_UFK} these are indeed norms.

Now, due to the presence of the gradient unknown in our model problem, we need to introduce a Sobolev
space where two variables are jointly controlled. For $T\in\cT$, we define the combined local test space
\(
    \HDivdiv{T}
\)
with components $(\XXi,\btau)$ of symmetric tensors $\XXi$ and vector functions $\btau$
as the completion (of smooth tensor and vector functions) with respect to the norm
\[
    \|(\XXi,\btau)\|_{\Divdiv,T}
    :=
    \bigl(\|\XXi\|_T^2 + \|\btau - \Div\XXi\|_T^2 + \|\div\btau\|_T^2\bigr)^{1/2}.
\]
The corresponding product space is $\HDivdiv{\cT}$ with norm $\|(\cdot,\cdot)\|_{\Divdiv,\cT}$.
Elements of this space can be tested with traces of $H^2$-spaces.
Specifically, for $T\in\cT$, we define the formally new trace operator
\begin{align} \label{trGG2T}
   \ttraceGG{T}:\;
   \left\{\begin{array}{cll}
      H^2(T) & \to & \HDivdiv{T}',\\
      z & \mapsto & \dual{\ttraceGG{T}(z)}{(\XXi,\btau)}_{\partial T}\\
      && := \vdual{\div\btau}{z}_T + \vdual{\btau-\Div\XXi}{\grad z}_T - \vdual{\XXi}{\Grad\grad z}_T
   \end{array}\right..
\end{align}
It is uniformly bounded for $T\in\cT$ if we provide $H^2(T)$ with the norm
$\|\cdot\|_{2,T}$ enriched with the $L_2$-norm of the gradient.
In any case, we will not use bounds of this operators but rather of the corresponding
global operator defined by
\begin{equation} \label{trGG2}
   \ttraceGG{}:\;
   \left\{\begin{array}{cll}
      H^2_0(\Omega) & \to & \HDivdiv{\cT}',\\
      z & \mapsto & \dual{\ttraceGG{}(z)}{(\XXi,\btau)}_\cS\\
    && := \vdual{\div\btau}{z}_\cT + \vdual{\btau-\Div\XXi}{\grad z}_\cT - \vdual{\XXi}{\Grad\grad z}.
   \end{array}\right.
\end{equation}
Like for the previous trace operator $\traceGG{}$, integration by parts reveals that the functional
$\ttraceGG{}(z)$ only depends on the traces of $z$ and $\grad z$ on $\cS$ (the union of the boundaries
of the elements $T\in\cT$). These traces are well defined in the standard way such that
$\ttraceGG{}\bigl(H^2_0(\Omega)\bigr)=\traceGG{}\bigl(H^2_0(\Omega)\bigr)=\bH^{3/2,1/2}_{00}(\cS)$.

Note that the norm in $\bH^{3/2,1/2}_{00}(\cS)$ is defined dual to the (quotient) norm in $\HdDiv{\cT}$,
but here we test with elements of $\HDivdiv{\cT}$.
Therefore, we switch to the trace norm
\[
   \|\tz\|_\trggrad{0,\cS} := \inf\{\|v\|_2;\; v\in H^2_0(\Omega),\ \traceGG{}(v)=\tz\}
\]
which, in any case, is identical to the other norm due to \cite[Proposition 9]{FuehrerHN_UFK},
\[
   \|\tz\|_{3/2,1/2,00,\cS} = \|\tz\|_\trggrad{0,\cS}\quad\forall\tz\in\bH^{3/2,1/2}_{00}(\cS).
\]
For $\tz\in \bH^{3/2,1/2}_{00}(\cS)$ and $(\XXi,\btau)\in\HDivdiv{\cT}$ we use the duality pairing
\begin{align} \label{dualityGG2}
   \dual{\tz}{(\XXi,\btau)}_\cS
   &:= \dual{\ttraceGG{}(z)}{(\XXi,\btau)}_\cS
   \quad\text{for}\ z\in H^2_0(\Omega)\text{ with } \ttraceGG{}(z)=\tz.
\end{align}
Let us note that the definitions \eqref{ipGG} and \eqref{trGG2} imply the identity
\[
   \dual{\ttraceGG{}(z)}{(\TTheta,\Div\TTheta)}_\cS
   =
   \dual{\traceGG{}(z)}{\TTheta}_\cS
   \quad\forall\TTheta\in\HdDiv{\cT}.
\]
Since
\[
   \|(\TTheta,\Div\TTheta)\|_{\Divdiv,\cT} = \|\TTheta\|_{\dDiv,\cT}\quad\forall\TTheta\in\HdDiv{\cT}
\]
it is therefore clear that $\ttraceGG{}$ is merely an extension of $\traceGG{}$
when identifying $\HdDiv{\cT}$ with a subspace of $\HDivdiv{\cT}$ through the map
$\TTheta\mapsto (\XXi,\btau):=(\TTheta,\Div\TTheta)$.
Furthermore, for sufficiently smooth functions $(\XXi,\btau)$, e.g., such that
$\Div\XXi|_T\in\bL_2(T)$ for $T\in\cT$,

\[
   \dual{\ttraceGG{T}(z)}{(\XXi,\btau)}_{\partial T}
   =
   \dual{\nn\cdot\btau}{z}_{\partial T}
   -
   \dual{\XXi\nn}{\grad z}_{\partial T}\quad\forall z\in H^2(T)
\]
where the boundary dualities are defined in the usual sense, cf.~\cite[Remark 1]{FuehrerHN_UFK}.

\subsection{Variational formulation} \label{sec_VF}

Now let us consider our problem \eqref{prob}.
Testing relation \eqref{p1} with $z\in H^2(\cT)$ and using \eqref{ipDD} gives
\begin{align} \label{test_divDivM}
   -\vdual{f}{z} = \vdual{\MM}{\Grad\grad z}_\cT + \dual{\traceDD{}(\MM)}{z}_\cS.
\end{align}
Now, testing relations \eqref{p3} and \eqref{p4} with independent functions of natural
regularity, this leads to a stability problem in the corresponding adjoint problem. In fact,
the unknown used to test equation \eqref{p4} has the same regularity issue as the shear force
$\Div\MM$ which is not in $\bL_2(\Omega)$ in general
(cf.~\cite{FuehrerHN_UFK} and our example in Section~\ref{sec_num}).
The combined test space $\HDivdiv{\cT}$ circumvents this problem.

Testing \eqref{p3} and \eqref{p4} with $\XXi$ and $\btau$, respectively, for
$(\XXi,\btau)\in \HDivdiv{\cT}$, using \eqref{p4} to replace $\Grad\btheta=\Grad\grad u$,
we obtain by \eqref{trGG2} (with $z$ replaced by $u$) the relation
\begin{align} \label{test_p34}
   \lefteqn{
    \vdual{\cCinv\MM}{\XXi} + \vdual{\Grad\grad u}{\XXi}
    +\vdual{\btheta}{\btau} - \vdual{\grad u}{\btau}
   }\nonumber\\
   &=
   \vdual{\cCinv\MM}{\XXi} + \vdual{\btheta}{\btau}
   -\vdual{\grad u}{\Div\XXi}_\cT + \vdual{u}{\div\btau}_\cT - \dual{\ttraceGG{}(u)}{(\XXi,\btau)}_\cS
   =0.
\end{align}
Since $\grad u=\btheta$, one sees that the terms
$\vdual{\btheta}{\btau}$ and $\vdual{\grad u}{\Div\XXi}_\cT$ are well defined
as the difference $\vdual{\btheta}{\btau-\Div\XXi}_\cT$ because $\btau-\pwDiv\XXi\in\bL_2(\Omega)$.
Now, combining \eqref{test_divDivM} and \eqref{test_p34}, we obtain
\begin{align*}
   &\vdual{\MM}{\Grad\grad z}_\cT + \dual{\traceDD{}(\MM)}{z}_\cS
    \\
    &+ \vdual{\cCinv\MM}{\XXi} +  \vdual{\btheta}{\btau-\Div\XXi}_\cT
     + \vdual{u}{\div\btau}_\cT -\dual{\ttraceGG{}(u)}{(\XXi,\btau)}_\cS
     = -\vdual{f}{z}.
\end{align*}
Finally we introduce independent trace variables $\tu:=\ttraceGG{}(u)$, $\tQ:=\traceDD{}(\MM)$,
define the ansatz and test spaces
\begin{align*}
   &\UU := L_2(\Omega)\times\bL_2(\Omega)\times\LL_2^s(\Omega)\times
           \bH^{3/2,1/2}_{00}(\cS)\times \bH^{-3/2,-1/2}(\cS),\\
   &\VV := H^2(\cT)\times\HDivdiv{\cT}
\end{align*}
with respective (squared) norms
\begin{align*}
   \|(u,\btheta,\MM,\tu,\tQ)\|_\UU^2
   &:=
   \|u\|^2 + \|\btheta\|^2 + \|\MM\|^2 + \|\tu\|_\trggrad{\cS,0}^2 + \|\tQ\|_{-3/2,-1/2,\cS}^2,
   \\
   \|(z,\XXi,\btau)\|_\VV^2
   &:=
   \|z\|_{2,\cT}^2 + \|(\XXi,\btau)\|_{\Divdiv,\cT}^2
   =
   \|z\|_{2,\cT}^2 + \|\XXi\|^2 + \|\btau-\Div\XXi\|_\cT^2 + \|\div\btau\|_\cT^2,
\end{align*}
and the following (bi)linear forms,
\begin{align} \label{b}
   b(u,\btheta,\MM,\tu,\tQ;z,\XXi,\btau)
   &:=
   \vdual{u}{\div\btau}_\cT + \vdual{\btheta}{\btau-\Div\XXi}_\cT
   + \vdual{\MM}{\cCinv\XXi+\Grad\grad z}_\cT \nonumber\\
   &\quad -\dual{\tu}{(\XXi,\btau)}_\cS + \dual{\tQ}{z}_\cS
   \\
   L(z,\XXi,\btau) &:= -\vdual{f}{z}.\nonumber
\end{align}
Then, our ultraweak variational formulation of \eqref{prob} is:
\emph{Find $(u,\btheta,\MM,\tu,\tQ)\in \UU$ such that}
\begin{align} \label{VF}
   b(u,\btheta,\MM,\tu,\tQ;z,\XXi,\btau) = L(z,\XXi,\btau)
   \quad\forall (z,\XXi,\btau)\in\VV.
\end{align}

\subsection{Well-posedness and DPG approximation} \label{sec_DPG}

Let us state one of our main results.

\begin{theorem} \label{thm_stab}
Let $f\in L_2(\Omega)$, or any functional $L\in\VV'$, be given. Then, \eqref{VF}
has a unique solution $(u,\btheta,\MM,\tu,\tQ)\in \UU$, and it satisfies
\[
   \|u\| + \|\btheta\| + \|\MM\| + \|\tu\|_\trggrad{\cS,0} + \|\tQ\|_{-3/2,-1/2,\cS}
   \lesssim
   \|f\|\quad\text{(or $\|L\|_{\VV'}$)}.
\]
The hidden constant is independent of $f$ (or $L$) and $\cT$.
\end{theorem}

A proof of this theorem will be given in Section~\ref{sec_pf}.

Now, the construction of the DPG method with optimal test functions for an approximation
of \eqref{VF} is standard. One chooses a finite-dimensional subspace
$\UU_h\subset\UU$ and selects test functions based on the trial-to-test operator
$\ttt:\;\UU\to\VV$, which is defined by
\[
   \ip{\ttt(\uu)}{\vv}_\VV = b(\uu,\vv)\quad\forall\vv\in\VV.
\]
Here, $\ip{\cdot}{\cdot}_\VV$ denotes the inner product in $\VV$,
\begin{align*}
   \ip{(z,\XXi,\btau)}{(\deltaz,\deltaXXi,\deltabtau)}_\VV
   &:=
   \vdual{z}{\deltaz} + \vdual{\Grad\grad z}{\Grad\grad\deltaz}_\cT
   \\
   &\quad +
   \vdual{\XXi}{\deltaXXi} + \vdual{\Div\XXi-\btau}{\Div\deltaXXi-\deltabtau}_\cT
   + \vdual{\div\btau}{\div\deltabtau}_\cT
\end{align*}
for $(z,\XXi,\btau),(\deltaz,\deltaXXi,\deltabtau)\in\VV$.

The discrete method is: \emph{Find $\uu_h\in\UU_h$ such that}
\begin{align} \label{DPG}
   b(\uu_h,\ttt\bdeltau) = L(\ttt\bdeltau) \quad\forall\bdeltau\in\UU_h.
\end{align}
This construction is equivalent to minimizing the residual $\|B(\uu-\uu_h)\|_{\VV'}$
where $B:\;\UU\to\VV'$ is the operator induced by the bilinear form $b(\cdot,\cdot)$.
When referring to DPG schemes, $\|\uu\|_\EE:=\|B\uu\|_{\VV'}$ is called the \emph{energy norm}.
Since the scheme minimizes the residual in the $\VV'$-norm, we obtain
the best approximation in the energy norm. For this and other results we refer to
the first papers on this subject by Demkowicz and Gopalakrishnan, e.g.,~\cite{DemkowiczG_11_ADM}.

If one knows that the energy norm is uniformly equivalent to the norm $\|\cdot\|_\UU$,
then the minimizing property immediately shows quasi-optimal convergence of the DPG scheme
in the $\UU$-norm. Indeed, this is the case so that we can state our second main result.

\begin{theorem} \label{thm_DPG}
Let $f\in L_2(\Omega)$ be given with solution $\uu\in\UU$ to \eqref{VF}.
Furthermore, let $\UU_h\subset\UU$ be a finite-dimensional subspace.
Then, \eqref{DPG} has a unique solution $\uu_h\in\UU_h$,
and it satisfies the quasi-optimal error estimate
\[
   \|\uu-\uu_h\|_\UU \lesssim \|\uu-\bw\|_\UU \quad\forall\bw\in\UU_h.
\]
The hidden constant is independent of $f$, $\cT$ and $\UU_h$.
\end{theorem}

A proof of this theorem will be given in Section~\ref{sec_pf}.

\subsection{Fully discrete scheme} \label{sec_DPG_discrete}

In practice, optimal test functions in scheme \eqref{DPG} have to be approximated.
This is done by replacing the test space $\VV$ by a finite-dimensional subspace
$\VV_h\subset\VV$ and the trial-to-test operator $\ttt:\;\UU\to\VV$ by the discretized operator
$\ttt_h:\;\UU_h\to\VV_h$ defined by
\[
   \ip{\ttt_h(\ww)}{\vv}_\VV = b(\ww,\vv)\quad\forall\vv\in\VV_h.
\]
The fully discrete scheme, called \emph{practical} DPG method by Gopalakrishnan and Qiu
\cite{GopalakrishnanQ_14_APD}, then reads:
\emph{Find $\uu_h\in\UU_h$ such that}
\begin{align} \label{DPG_discrete}
   b(\uu_h,\ttt_h\bdeltau) = L(\ttt_h\bdeltau) \quad\forall\bdeltau\in\UU_h.
\end{align}
The discrete stability and quasi-optimal convergence of \eqref{DPG_discrete} follows
from the existence of a Fortin operator $\Pi:\;\VV\to\VV_h$ satisfying
\begin{subequations} \label{Fortin}
\begin{align}
   &b(\ww,\vv-\Pi\vv) = 0\quad\forall\ww\in\UU_h,\ \forall\vv\in\VV,\\
           &\|\Pi\vv\|_\VV \le C_\Pi\|\vv\|_\VV\quad\forall\vv\in\VV,
\end{align}
\end{subequations}
with a positive constant $C_\Pi$ that is independent of the underlying mesh $\cT$.
To construct a Fortin operator with these properties one usually makes additional assumptions
on the mesh sequence (shape-regularity of elements) and bounds the polynomial degrees.
Theorem~2.1 from \cite{GopalakrishnanQ_14_APD} then proves the following result.

\begin{lemma} \label{la_Fortin}
Assume that there exists a Fortin operator $\Pi:\;\VV\to\VV_h$ satisfying \eqref{Fortin}.
Then, \eqref{DPG_discrete} is well posed and converges quasi-optimally,
\[
   \|\uu-\uu_h\|_\UU \lesssim \|\uu-\bw\|_\UU \quad\forall\bw\in\UU_h
\]
where $\uu_h$ denotes the solution of \eqref{DPG_discrete}.
\end{lemma}

For our fully discrete scheme we consider lowest-order approximations
and meshes with shape-regular triangular elements.
We will construct a Fortin operator for this case. Let us start defining the approximation space.

\paragraph{Approximation space.}

We use the construction of discrete subspaces of
$\bH^{3/2,1/2}_{00}(\cS)$ and $\bH^{-3/2,-1/2}(\cS)$ from \cite{FuehrerHN_UFK}.
Specifically, we restrict our consideration to two dimensions (that is, the plate problem),
$\di=2$, and use regular triangular meshes $\cT$ of shape-regular elements.

For $T\in\cT$, let $P^p(T)$ denote the space of polynomials on $T$ which are
of order $\le p\in\N_0$, and define
\begin{align*}
  P^p(\cT) := \{v\in L_2(\Omega);\; v|_T\in P^p(T) \,\forall T\in\cT\}.
\end{align*}
Setting $\bP^p(T):=P^p(T)^2$, $\bP^p(\cT):=P^p(\cT)^2$,
$\PP^{p,s}(T):=P^p(T)^{2\times 2}\cap\LL_2^s(T)$, and
$\PP^{p,s}(\cT) := P^p(\cT)^{2\times 2}\cap\LL_2^s(\Omega)$,
we approximate $(u,\btheta,\MM)\in L_2(\Omega)\times\bL_2(\Omega)\times\LL_2^s(\Omega)$ by
\begin{align*}
  (u_h,\btheta_h,\MM_h) \in P^0(\cT) \times \bP^0(\cT) \times \PP^{0,s}(\cT).
\end{align*}
It remains to introduce discrete spaces for the skeleton variables $(\tu,\tQ)$.
As they are traces, basis functions have to satisfy certain conformity conditions.
This is why we need to introduce some notation for edges.

For $T\in\cT$ let $\cE_T$ denote the set of its edges, and $\cE := \bigcup_{T\in\cT} \cE_T$.
Denoting the space of polynomials up to degree $p$ on $E\in\cE$ by $P^p(E)$, we define
\begin{align*}
  P^p(\cE_T) := \{v\in L_2(\partial T);\; v|_E \in P^p(E)\ \forall E\in \cE_T\},
  \quad T\in\cT.
\end{align*}
Then we define, for $T\in\cT$, the local space
\begin{align} \label{hatUT}
  \widehat U_T
  :=
  \traceGG{T}\left(\{v\in H^2(T);\; \Delta^2 v + v = 0, \, v|_{\partial T} \in P^3(\cE_T), \, 
  \nn\cdot\grad v|_{\partial T}\in P^1(\cE_T)\}\right).
\end{align}
We denote by $\cN_T$ the set of vertices of $T\in\cT$,
set $\cN := \bigcup_{T\in\cT} \cN_T$ and denote by $\cN_0\subset \cN$ the set of nodes
which are not on $\Gamma$. Our discrete subspace of $\bH^{3/2,1/2}_{00}(\cS)$ then is
\begin{align*}
  \widehat U_\cS
  :=
  \{\tv\in \bH^{3/2,1/2}_{00}(\cS);\;
    \tv|_{\partial T}  \in \widehat U_T\ \forall T\in\cT\}
\end{align*}
with associated degrees of freedom $\{(v(e),\grad v(e));\; e\in\cN_0\}$.
The space is of dimension $3\#\cN_0$.

Now, to construct a discrete subspace of $\bH^{-3/2,-1/2}(\cS)$, we define the local space
\begin{align} \label{UdDivT}
\begin{split}
  U_{\trddiv{T}} := \{\TTheta\in &\,\HdDiv{T};\; \Grad\grad\div\Div \TTheta +\TTheta = 0,\\
  &\bigl(\nn\cdot\Div\TTheta
         + \partial_{\bt,\cE_T}
                    (\bt\cdot\TTheta\nn)\bigr)|_{\partial T}\in P^0(\cE_T),\quad
  \nn\cdot\TTheta\nn|_{\partial T} \in P^0(\cE_T) \}
\end{split}
\end{align}
for $T\in\cT$.
Here, $\partial_{\bt,\cE_T}$ denotes the tangential derivative operator that is taken piecewise
on the edges of $\partial T$, cf.~\cite[Remark~7]{FuehrerHN_UFK}.
The space $U_{\trddiv{T}}$ has the following degrees of freedom,
\begin{subequations}\label{eq:dofDDlocal}
\begin{align}
  &\alpha_E := \dual{\nn\cdot\Div\TTheta + \partial_\bt (\bt\cdot\TTheta\nn)}{1}_E
  \hspace{-7em}&&(E\in\cE_T),\label{eq:dofDDlocal_a}\\
  &\beta_E := \dual{\nn\cdot\TTheta\nn}{1}_E
  &&(E\in\cE_T),\label{eq:dofDDlocal_b}\\
  &\gamma_e := \jjump{\TTheta}_{\partial T}(e)
  &&(e\in \cN_T).\label{eq:dofDDlocal_c}
\end{align}
\end{subequations}
Here, $\jjump{\TTheta}_{\partial T}(e)$ denotes the jump of the trace
$\bt\cdot\TTheta\nn|_{\partial T}$ at the vertex $e$ in mathematically positive orientation.
Now, the corresponding global space is
\begin{align*}
  U_{\trddiv\cT}
  :=
  \{\TTheta\in \HdDiv{\Omega};\; \TTheta|_{T}\in U_{\trddiv{T}}\ \forall T\in\cT\}.
\end{align*}
According to \cite[Lemma 17]{FuehrerHN_UFK}, it has the degrees of freedom
\begin{align}
  \dual{\nn\cdot\Div\TTheta + \partial_\bt (\bt\cdot\TTheta\nn)}{1}_E
  &\quad (E\in\cE),
  \nonumber\\
  \dual{\nn\cdot\TTheta\nn}{1}_E
  &\quad (E\in\cE),
  \nonumber\\
  \jjump{\TTheta}_{\partial T}(e)
  &\quad (e\in \cN_T,\  T\in\cT),
  \nonumber\\
  \text{subject to }
  \sum_{T\in \omega(e)} \jjump{\TTheta}_{\partial T}(e) = 0
  &\quad\forall e\in\cN_0,
  \label{eq:dofDDglobal:d}
\end{align}
that is, its dimension is $\#\cE+\#\cE+3\#\cT-\#\cN_0$.
The set $\omega(e)$ consists of the elements $T\in\cT$ which have $e$ as a vertex.
The constraints~\eqref{eq:dofDDglobal:d} can be implemented by using Lagrangian multipliers.
Now, for the approximation of $\tQ\in\bH^{-3/2,-1/2}(\cS)$, we use the trace space
\begin{align*}
  \widehat Q_{\cS} := \traceDD{} (U_{\trddiv\cT}).
\end{align*}
It has the same degrees of freedom as $U_{\trddiv\cT}$, cf.~\cite{FuehrerHN_UFK}.

Eventually, our discrete subspace $\UU_h\subset\UU$ for the DPG approximation is
\begin{equation} \label{Uh}
  \UU_h := P^0(\cT) \times \bP^0(\cT) \times \PP^{0,s}(\cT) \times \widehat U_\cS \times \widehat Q_\cS.
\end{equation}
By \cite[Theorem~19]{FuehrerHN_UFK}, this space yields an approximation order $O(h)$ for sufficiently
smooth solutions.

\paragraph{Test space and quasi-optimal convergence.}

To define the fully discrete DPG scheme \eqref{DPG_discrete}, it remains to select a discrete test space
$\VV_h\subset\VV$ that allows for the construction of a Fortin operator. We select piecewise polynomial
spaces of degrees three and four,
\begin{equation} \label{Vh}
   \VV_h := P^3(\cT) \times \PP^{4,s}(\cT)\times \bP^3(\cT).
\end{equation}
This selection guarantees the well-posedness and quasi-optimal convergence of the discrete scheme.

\begin{theorem} \label{thm_DPG_discrete}
With the selection \eqref{Uh} and \eqref{Vh} of $\UU_h$ and $\VV_h$, respectively, the
scheme \eqref{DPG_discrete} is well posed and converges quasi-optimally,
\[
   \|\uu-\uu_h\|_\UU \lesssim \|\uu-\bw\|_\UU \quad\forall\bw\in\UU_h
\]
with $\uu_h$ denoting the solution of \eqref{DPG_discrete}.
\end{theorem}

\begin{proof}
By Lemma~\ref{la_Fortin}, it is enough to construct a Fortin operator $\Pi:\;\VV\to\VV_h$
that is bounded and satisfies
\(
   b(\ww,\vv-\Pi(\vv)) = 0
\)
$\forall \ww\in\UU_h$. Constructing $\Pi$ component-wise, 
\[
   \Pi\vv=(\projGG(z),\projD(\XXi,\btau))\quad\text{for}\
   \vv=\bigl(z,(\XXi,\btau)\bigr)\in H^2(\cT)\times\HDivdiv{\cT},
\]
it is enough to show that $\projGG:\;H^2(\cT)\to P^3(\cT)$ and
$\projD=(\projDQ,\projDtau):\;\HDivdiv{\cT}\to \PP^{4,s}(\cT)\times\bP^3(\cT)$ are bounded and satisfy
\begin{subequations} \label{pf_DPG_discrete_1}
\begin{alignat}{2}
   &\dual{\tQ_h}{z}_\cS = \dual{\tQ_h}{\projGG(z)}_\cS \quad&& \forall\tQ_h\in \widehat Q_\cS,\\
   &\vdual{\MM_h}{\Grad\grad z}_\cT = \vdual{\MM_h}{\Grad\grad \projGG(z)}_\cT
   \quad&& \forall \MM_h\in \PP^{0,s}(\cT)
\end{alignat}
\end{subequations}
for any $z\in H^2(\cT)$, and, since $\cC$ is self-adjoint and maps $\PP^{0,s}(\cT)\to \PP^{0,s}(\cT)$,
\begin{alignat}{2}
   &\dual{\tu_h}{(\XXi,\btau)}_\cS = \dual{\tu_h}{\projD(\XXi,\btau)}_\cS
   \quad&& \forall \tu_h\in \widehat U_\cS,\nonumber\\
   &\vdual{\MM_h}{\XXi}_\cT = \vdual{\MM_h}{\projDQ(\XXi,\btau)}_\cT
   \quad&& \forall \MM_h\in \PP^{0,s}(\cT),\nonumber\\
   \label{pf_DPG_discrete_2}
   &\vdual{\btheta_h}{\btau-\Div\XXi}_\cT = \vdual{\btheta_h}{\projDtau(\XXi,\btau)-\Div\projDQ(\XXi,\btau)}_\cT
   \quad&& \forall \btheta_h\in \bP^0(\cT),\\
   &\vdual{u_h}{\div\btau}_\cT = \vdual{u_h}{\div\projDtau(\XXi,\btau)}_\cT
   \quad&& \forall u_h\in P^0(\cT)
   \label{pf_DPG_discrete_3}
\end{alignat}
for any $(\XXi,\btau)\in\HDivdiv{\cT}$.
By definition of the bilinear form $b(u,\btheta,\MM,\tu,\tQ;z,\XXi,\btau)$, the required
orthogonality then follows. Now, the operators $\projGG$ and $\projD$ will be constructed in
\S\S\ref{sec_projGG} and~\ref{sec_projD} below. Specifically, the respective mapping and
orthogonality properties are shown by Lemmas~\ref{la_projGG} and \ref{la_projD}.
\end{proof}

\subsection{Remark on the fully discrete scheme without gradient variable} \label{sec_DPG_discrete2}

In Section~\ref{sec_projDD} below we also construct a Fortin operator $\projDD$ for the space
$\HdDiv{\cT}$. It ensures that the lowest-order DPG scheme from \cite{FuehrerHN_UFK}
(for the Kirchhoff--Love plate bending problem without unknown $\btheta=\grad u$)
is well posed and converges quasi-optimally when selecting the discrete test space
\[
   \wilde\VV_h=P^3(\cT)\times \PP^{4,s}(\cT) \subset \wilde\VV := H^2(\cT)\times\HdDiv{\cT}.
\]
The numerical results in \cite{FuehrerHN_UFK} suggest that the smaller discrete space
\(
   P^3(\cT)\times \PP^{2,s}(\cT)
\)
is sufficient for the considered examples to guarantee discrete stability.

To be more specific let us recall the bilinear form of the variational formulation
 from \cite{FuehrerHN_UFK}. It is
\begin{align*}
   \tilde b(u,\MM,\tu,\tQ;z,\TTheta) :=
   &\vdual{\MM}{\pwGrad\pwgrad z+\cCinv\TTheta}
   +\vdual{u}{\pwdiv\pwDiv\TTheta}
   + \dual{\tQ}{z}_\cS - \dual{\tu}{\TTheta}_\cS,
\end{align*}
with $u$, $\MM$, $\tu$ and $\tQ$ as in this paper,
and test functions $(z,\TTheta)$ are taken in $\wilde\VV=H^2(\cT)\times\HdDiv{\cT}$.
Now, as in the proof of Theorem~\ref{thm_DPG_discrete}, one sees that the fully discrete
DPG scheme in this case, with approximation space
\[
  \wilde\UU_h := P^0(\cT) \times \PP^{0,s}(\cT) \times \widehat U_\cS \times \widehat Q_\cS
\]
and discrete test space $\wilde\VV_h$ as specified before, is well posed and converges quasi-optimally if there
is a Fortin operator $\wilde\Pi:\;\wilde\VV\to\wilde\VV_h$ that is uniformly continuous and
satisfies $\tilde b(\wilde\uu;\wilde\vv-\wilde\Pi\wilde\vv)=0$ for any
$\wilde\vv=(z,\TTheta)\in\wilde\VV$ and $\wilde\uu\in\wilde\UU_h$. Using the Fortin operator
$\projDD:\;\HdDiv{\cT}\to\wilde\VV_h$ from Section~\ref{sec_projDD} we define
$\wilde\Pi:=(\projGG,\projDD)$ and see that it satisfies the required orthogonality properties.
Indeed, one set of relations is satisfied by \eqref{pf_DPG_discrete_1}, and the remaining
relations (again using that $\cC$ induces a self-adjoint isomorphism
$\PP^{0,s}(\cT)\to \PP^{0,s}(\cT)$)
\begin{alignat}{2}
   &\dual{\tu_h}{\TTheta}_\cS = \dual{\tu_h}{\projDD\TTheta}_\cS
   \quad&& \forall \tu_h\in \widehat U_\cS,\nonumber\\
   &\vdual{\MM_h}{\TTheta}_\cT = \vdual{\MM_h}{\projDD\TTheta}_\cT
   \quad&& \forall \MM_h\in \PP^{0,s}(\cT),\nonumber\\
   &\vdual{u_h}{\div\Div\TTheta}_\cT = \vdual{u_h}{\div\Div\projDD\TTheta}_\cT
   \quad&& \forall u_h\in P^0(\cT)
   \label{orth_Fortin_dDiv_u}
\end{alignat}
for any $\TTheta\in\HdDiv{\cT}$ hold by Lemma~\ref{la_projDD} in Section~\ref{sec_projDD}.
Both components $\projGG$ and $\projDD$ of $\wilde\Pi$ are also uniformly continuous
by Lemmas~\ref{la_projGG} and~\ref{la_projDD}, respectively.

\section{Analysis of the adjoint problem and proofs of Theorems~\ref{thm_stab},\ref{thm_DPG}}
\label{sec_adj}

The well-posedness of the ultraweak formulation \eqref{VF} is equivalent to that of its
adjoint problem. In order to formulate this problem we have to identify the
functionals that are induced by the skeleton terms $\dual{\tu}{(\XXi,\btau)}_\cS$
on $(\XXi,\btau)\in\HDivdiv{\cT}$ and $\dual{\tQ}{z}$ on $z\in H^2(\cT)$.

As we have seen in \cite{FuehrerHN_UFK}, every element $\tQ$ of the trace space
$\bH^{-3/2,-1/2}(\cS)=\traceDD{}\bigl(\HdDiv{\Omega}\bigr)$
assigns a ``jump'' to $z\in H^2(\cT)$ through \eqref{dualityDD}, i.e.,
\(
   \dual{\tQ}{z}_\cS = \dual{\traceDD{}(\TTheta)}{z}_\cS
\)
with $\TTheta\in\HdDiv{\Omega}$ such that $\traceDD{}(\TTheta)=\tQ$.
The notation from \cite{FuehrerHN_UFK} is
\begin{align*}
   \jump{\cdot,\pwgrad\,\cdot}:\;
   \left\{\begin{array}{cll}
      H^2(\cT) & \to & \Bigl(\bH^{-3/2,-1/2}(\cS)\Bigr)'\\
      z        & \mapsto & \jump{z,\pwgrad z}(\tQ) := \dual{\tQ}{z}_\cS
   \end{array}\right.
\end{align*}
with operator norm denoted by $\|\cdot\|_{(-3/2,-1/2,\cS)'}$.

In \cite{FuehrerHN_UFK} we have also seen that every element
$\tv\in\bH^{3/2,1/2}_{00}(\cS)=\traceGG{}\bigl(H^2_0(\Omega)\bigr)$ assigns a value to the
``jump'' of $\TTheta\in\HdDiv{\cT}$.
In the formulation under consideration, however, the variable $\tu\in\bH^{3/2,1/2}_{00}(\cS)$
is acting on function pairs $(\XXi,\btau)\in\HDivdiv{\cT}$ through the duality \eqref{dualityGG2}.
We therefore define a new jump functional by
\begin{align} \label{jump_divDiv}
   \jump{(\cdot)\nn,\nn\cdot(\cdot)}:\;
   \left\{\begin{array}{cll}
      \HDivdiv{\cT} & \to & \Bigl(\bH^{3/2,1/2}_{00}(\cS)\Bigr)'\\
      (\XXi,\btau)  & \mapsto & \jump{\XXi\nn,\nn\cdot\btau}(\tv) := \dual{\tv}{(\XXi,\btau)}_\cS
   \end{array}\right.
\end{align}
measured in the operator norm $\|\cdot\|_{(\trggrad{\cS,0})'}$.

Now, considering the duality pairings appearing in the bilinear form \eqref{b},
the adjoint problem of \eqref{VF} is as follows.

\emph{Find $z\in H^2(\cT)$ and $(\XXi,\btau)\in\HDivdiv{\cT}$ such that}
\begin{subequations} \label{adj}
\begin{alignat}{2}
    \pwdiv\btau                          &= g    && \ \in L_2(\Omega)\label{a1},\\
    \btau - \pwDiv\XXi                   &= \bl  && \ \in \bL_2(\Omega)\label{a2},\\
    \cCinv\XXi + \pwGrad\pwgrad z        &= \HH  && \ \in \LL_2^s(\Omega)\label{a3},\\
    \jump{\XXi\nn,\nn\cdot\btau}         &= \br  && \ \in \Bigl(\bH^{3/2,1/2}_{00}(\cS)\Bigr)'\label{jtauXi},\\
    \jump{z,\pwgrad z}          &= \bj  && \ \in \Bigl(\bH^{-3/2,-1/2}(\cS)\Bigr)'\label{jz}.
\end{alignat}
\end{subequations}
Here, the differential operators with index $\cT$ refer to the operators acting on the
corresponding product spaces (they are taken piecewise with respect to $\cT$).

To prove the well-posedness of \eqref{adj} we proceed as in \cite{FuehrerHN_UFK} and study its reduced form.
We first consider linear combinations of the relations, after testing appropriately.
Specifically, for $\deltaz\in H^2_0(\Omega)$, we test \eqref{a1}--\eqref{a3}, respectively, by
$\deltaz$, $\grad\deltaz$, and $-\cC\Grad\grad\deltaz$. Summation yields
\begin{align} \label{sum_adj}
   \vdual{g}{\deltaz} + \vdual{\bl}{\grad\deltaz} - \vdual{\cC\HH}{\Grad\grad\deltaz}
   =
     \vdual{\pwdiv\btau}{\deltaz} + \vdual{\btau-\pwDiv\XXi}{\grad\deltaz} 
   - \vdual{\XXi+\cC\pwGrad\pwgrad z}{\Grad\grad\deltaz}
\end{align}
Now, by \eqref{trGG2}, \eqref{dualityGG2}, \eqref{jump_divDiv}, and \eqref{jtauXi},
\begin{align*}
   &\vdual{\pwdiv\btau}{\deltaz} + \vdual{\btau-\pwDiv\XXi}{\grad\deltaz}
   - \vdual{\XXi}{\Grad\grad\deltaz}
   =
   \dual{\ttraceGG{}(\deltaz)}{(\XXi,\btau)}_\cS
   \\
   &=
   \jump{\XXi\nn,\nn\cdot\btau}(\ttraceGG{}(\deltaz))
   =
   \br(\ttraceGG{}(\deltaz)).
\end{align*}
Therefore, from \eqref{sum_adj} we obtain the following variational form of the reduced adjoint problem.

\emph{Given $g\in L_2(\Omega)$, $\bl\in\bL_2(\Omega)$, $\HH\in\LL_2^s(\Omega)$,
$\br\in\bigl(\bH^{3/2,1/2}_{00}(\cS)\bigr)'$, and $\bj\in\Bigl(\bH^{-3/2,-1/2}(\cS)\Bigr)'$
find $z\in H^2(\cT)$ such that}
\begin{subequations} \label{saddle}
\begin{alignat}{2}
   \label{s1}
   \vdual{\cC\pwGrad\pwgrad z}{\Grad\grad\deltaz}
   &=
   \vdual{\cC\HH}{\Grad\grad\deltaz} - \vdual{\bl}{\grad\deltaz} - \vdual{g}{\deltaz}
   + \br(\ttraceGG{}(\deltaz))
   &&\quad\forall\deltaz\in H^2_0(\Omega),\\
   \label{s2}
   \jump{z,\pwgrad z}(\bdeltaq) &= \bj(\bdeltaq)
   &&\quad\forall\bdeltaq\in \bH^{-3/2,-1/2}(\cS).
\end{alignat}
\end{subequations}

At the heart of our analysis is the well-posedness of \eqref{saddle} whose proof requires
some tools developed in \cite{FuehrerHN_UFK}. These tools are recalled next.

\begin{prop} \label{prop_tools} {\rm\cite[Propositions 8(i), 10]{FuehrerHN_UFK}}\\
(i) For $z\in H^2(\cT)$ it holds
\[
   z\in H^2_0(\Omega)\quad\Leftrightarrow\quad
   \dual{\bq}{z}_\cS=0 \quad\forall\bq\in \bH^{-3/2,-1/2}(\cS).
\]
(ii) It holds
\[
   \|z\|_{2,\cT}
   \lesssim
   \sup_{\deltaz\in H^2_0(\Omega)\setminus\{0\}}
   \frac {\vdual{\pwGrad\pwgrad z}{\cC\Grad\grad\deltaz}}{\|\cC\Grad\grad\deltaz\|}
   +
   \|\jump{z,\pwgrad z}\|_{(-3/2,-1/2,\cS)'}
   \quad\forall z\in H^2(\cT)
\]
with an implicit constant that is independent of the mesh $\cT$.
\end{prop}

Now we can state and prove the well-posedness of the reduced adjoint problem.

\begin{lemma} \label{la_saddle}
Problem \eqref{saddle} has a unique solution $z\in H^2(\cT)$. It satisfies
\[
   \|z\|_{2,\cT}
   \lesssim
   \|g\| + \|\bl\| + \|\HH\| + \|\br\|_{(\trggrad{\cS,0})'} + \|\bj\|_{(-3/2,-1/2,\cS)'}.
\]
\end{lemma}

\begin{proof}
With the appropriate tools at hand, the proof of this lemma is identical to the proof of
Lemma~13 in \cite{FuehrerHN_UFK}. For the convenience of the reader let us recall the principal steps.
Adding relations \eqref{s1}, \eqref{s2} we represent \eqref{saddle} with the notation
\[
   a(z;\deltaz,\bdeltaq) = l(\deltaz,\bdeltaq).
\]
The boundedness of the right-hand side functional is immediate by the involved dualities.
The boundedness of the bilinear form $a$ is also clear.
It remains to check the two inf-sup conditions.

(i) $a(z;\deltaz,\bdeltaq)=0$ $\forall z\in H^2(\cT)$ implies $\deltaz=0$ and $\bdeltaq=0$.
Indeed, selecting $z:=\deltaz\in H^2_0(\Omega)$ so that $\jump{z,\pwgrad z}(\bdeltaq)=0$
by Proposition~\ref{prop_tools}(i), this shows that $\|\Grad\grad\deltaz\|=0$
by the positive definiteness of $\cC$. Therefore, $\deltaz=0$.
Now, using that $\deltaz=0$, we have that
\[
   \jump{z,\pwgrad z}(\bdeltaq) = 0\quad\forall z\in H^2(\cT),
\]
that is, $\|\bdeltaq\|_{-3/2,-1/2,\cS}=0$, hence $\bdeltaq=0$, cf.~definition \eqref{def_norm_trDD} of the norm.

(ii) The inf-sup condition
\[
   \sup_{(\deltaz,\bdeltaq)\in H^2_0(\Omega)\times\bH^{-3/2,-1/2}(\cS)\setminus\{0\}}
   \frac {\vdual{\pwGrad\pwgrad z}{\cC\Grad\grad\deltaz} + \jump{z,\pwgrad z}(\bdeltaq)}
         {\bigl(\|\cC\Grad\grad\deltaz\|^2 + \|\bdeltaq\|_{-3/2,-1/2,\cS}^2\bigr)^{1/2}}
   \gtrsim
   \|z\|_{2,\cT}\quad\forall z\in H^2(\cT)
\]
follows by duality, Proposition~\ref{prop_tools}(ii) and the norm equivalence
$\|\cC\Grad\grad\deltaz\|\simeq\|\deltaz\|_2$ for $\deltaz\in H^2_0(\Omega)$.
This finishes the proof of the lemma.
\end{proof}

\begin{prop} \label{prop_adj}
For arbitrary $g\in L_2(\Omega)$, $\bl\in\bL_2(\Omega)$, $\HH\in\LL_2^s(\Omega)$,
$\br\in\bigl(\bH^{3/2,1/2}_{00}(\cS)\bigr)'$, and $\bj\in\Bigl(\bH^{-3/2,-1/2}(\cS)\Bigr)'$
the adjoint problem \eqref{adj} has a unique solution $(z,\XXi,\btau)\in\VV$. It satisfies
\[
   \|z\|_{2,\cT} + \|(\XXi,\btau)\|_{\Divdiv,\cT}
   \lesssim
   \|g\| + \|\bl\| + \|\HH\| + \|\br\|_{(\trggrad{\cS,0})'} + \|\bj\|_{(-3/2,-1/2,\cS)'}.
\]
\end{prop}

\begin{proof}
By construction, the $z$-component of any solution $(z,\XXi,\btau)$ to the adjoint problem
\eqref{adj} satisfies the reduced adjoint problem \eqref{saddle},
which is uniquely solvable by Lemma~\ref{la_saddle}.
Then, $\XXi$ and $\btau$ are uniquely defined by \eqref{a3} and \eqref{a2}, respectively.
It is also easy to check that $\btau$ satisfies \eqref{a1} and
$(\XXi,\btau)$ satisfies \eqref{jtauXi}.

Finally, the bound for $\|z\|_{2,\cT}$ is provided by Lemma~\ref{la_saddle}.
Bounding the remaining norms is immediate:
$\|\XXi\|\le \|\cC\HH\| + \|\cC\Grad\grad z\|_\cT \lesssim \|\HH\| + \|\Grad\grad z\|_\cT$,
$\|\btau-\Div\XXi\|_\cT=\|\bl\|$, and $\|\div\btau\|_\cT=\|g\|$.
\end{proof}

\subsection{Proofs of Theorems~\ref{thm_stab},~\ref{thm_DPG}} \label{sec_pf}

To prove Theorem~\ref{thm_stab} we check the standard conditions.
\begin{enumerate}
\item {\bf Boundedness of the functional.} This is immediate since, for $f\in L_2(\Omega)$, it holds
      $L(z)\le \|f\|\,\|z\|\le \|f\|\,\|z\|_{2,\cT}$
      for any $(z,\XXi,\btau)\in\VV$.
\item {\bf Boundedness of the bilinear form.} 
      The bound $b(\uu,\vv)\lesssim \|\uu\|_\UU \|\vv\|_\VV$ for all $\uu\in\UU$ and
      $\vv\in\VV$ is also immediate by definition of the
      norms in $\UU$ and $\VV$, cf. the corresponding functional spaces in
      \eqref{a1}--\eqref{jtauXi}.
\item {\bf Injectivity.} If $\uu\in\UU$ with $b(\uu,\vv)=0\ \forall \vv\in\VV$ then
      $\uu=0$, as can be seen as follows. For given $\uu=(u,\btheta,\MM,\tu,\tQ)\in\UU$
      we select $g=u$, $\bl=\btheta$, and $\HH=\MM$, and let $\br\in (\bH^{3/2,1/2}_{00}(\cS))'$,
      and $\bj\in(\bH^{-3/2,-1/2}(\cS))'$
      be the Riesz functionals of $-\tu\in \bH^{3/2,1/2}_{00}(\cS)$
      and $\tQ\in\bH^{-3/2,-1/2}(\cS)$, respectively.
      According to Proposition~\ref{prop_adj}, there exists $\vv\in\VV$ that satisfies
      the adjoint problem \eqref{adj} with these functionals. It also yields
      \[
         b(\uu,\vv) = \|u\|^2 + \|\btheta\|^2 + \|\MM\|^2
                    + \|\tu\|_{\trggrad{\cS,0}}^2 + \|\tQ\|_{-3/2,-1/2,\cS}^2
         =0,
      \]
      which proves that $\uu=0$.
\item {\bf Inf-sup condition.} For given $\vv=(z,\XXi,\btau)\in\VV$ let $g$, $\bl$, $\HH$,
      $\br$, and $\bj$ be defined by the corresponding left-hand sides in \eqref{adj}.
      Then, by Proposition~\ref{prop_adj},
\begin{align*}
   \sup_{0\not=\uu\in\UU} \frac {b(\uu,\vv)}{\|\uu\|_\UU}
   &=
   \Bigl(\|g\|^2 + \|\bl\|^2 + \|\HH\|^2
         + \|\br\|_{(\trggrad{\cS,0})'}^2 + \|\bj\|_{(-3/2,-1/2,\cS)'}^2 \Bigr)^{1/2}
   \gtrsim
   \|\vv\|_{\VV}
\end{align*}
   with an implicit constant that is independent of $\vv$ and $\cT$.
\end{enumerate}
This proves Theorem~\ref{thm_stab}.

Recall that the DPG method delivers the best approximation in the energy norm $\|\cdot\|_\EE$,
\[
   \|\uu-\uu_h\|_\EE = \min\{\|\uu-\bw\|_\EE;\; \bw\in\UU_h\}.
\]
Therefore, to show Theorem~\ref{thm_DPG}, it is enough to prove the equivalence of the
energy norm and the norm $\|\cdot\|_\UU$.
The bound $\|\uu\|_\EE\lesssim\|\uu\|_\UU$ is equivalent to the boundedness
of $b(\cdot,\cdot)$, which we have just checked.
By definition of $\|\cdot\|_\EE=\|B(\cdot)\|_{\VV'}$, the other inequality,
$\|\uu\|_\UU\lesssim \|\uu\|_\EE$ for all $\uu\in\UU$, is equivalent to
the stability of the adjoint problem \eqref{adj}, which has been shown by
Proposition~\ref{prop_adj}. We have thus shown Theorem~\ref{thm_DPG}.

\def\auxGG{\Pi_0^{\mathrm{Ggrad}}} 
\def\Amat{\boldsymbol{A}}
\def\Bmat{\boldsymbol{B}}
\def\Cmat{\boldsymbol{C}}
\def\Tref{\widehat{T}}
\def\xref{\widehat{x}}
\def\zref{\widehat{z}}
\def\Qref{\widehat\QQ}
\def\Xiref{\widehat\XXi}
\def\Mref{\widehat\MM}
\def\auxDD{\widehat\Pi^{\mathrm{dDiv}}} 
\def\uref{\widehat{u}}
\def\BB{\boldsymbol{B}}
\def\trafo{\mathcal{H}}
\def\piola{\mathcal{P}}
\def\btauref{\widehat\btau}
\def\bthetaref{\widehat\btheta}

\def\basq{\boldsymbol{q}}
\def\basqref{\widehat{\boldsymbol{q}}}

\def\projDref{\widehat\Pi^{\mathrm{Div,div}}} 
\def\projDQref{\widehat\Pi^{\mathrm{Div}}}
\def\projDtauref{\widehat\Pi^{\mathrm{div}}}

\section{Fortin operators} \label{sec_Fortin}

In this section we construct and analyze Fortin operators for the lowest-order trial space
$\UU_h$, cf.~\eqref{Uh}. We also present an operator that is appropriate for the DPG scheme
from \cite{FuehrerHN_UFK} with test space $H^2(\cT)\times\HdDiv{\cT}$.
In the following section we start by defining transformations of the involved spaces
$H^2(T)$, $\HdDiv{T}$, $\HDivdiv{T}$ from the reference element
$\Tref:=\mathrm{conv}\{(0,0),(1,0),(0,1)\}$ to an element $T\in\cT$.
These transformations and their properties are valid in two and three space dimensions,
but are only given in $\R^2$ for simplicity.
Afterwards we present and analyze three Fortin operators, for $H^2(\cT)$ in \S\ref{sec_projGG},
for $\HdDiv{\cT}$ in \S\ref{sec_projDD}, and for $\HDivdiv{\cT}$ in \S\ref{sec_projD}.
A composition of the former and the latter yield the required operator in this paper,
as indicated in the proof of Theorem~\ref{thm_DPG_discrete}. All these operators are only
studied in two space dimensions.

\subsection{Transformations} \label{sec_trafo}

In the following, geometric objects, functions and differential operators carry the symbol
 ``\,$\widehat{\ }$\;'' when referring to objects related to the reference element $\Tref$.
Traces of functions are denoted without this symbol from now on, except for their
transformed functions on the boundary of $\Tref$.

Let $F_T:\;\Tref\to T$ denote the affine mapping
\begin{align*}
  \begin{pmatrix}\widehat x\\ \widehat y\end{pmatrix}
  \mapsto
  \BB_T
  \begin{pmatrix}\widehat x\\ \widehat y\end{pmatrix} + 
  \boldsymbol{a}_T,
\end{align*}
where $\boldsymbol{a}_T\in\R^2$, $\BB_T\in\R^{2\times 2}$, and set $J_T=\det(\BB_T)$.
In the following, we only consider transformations which generate families of shape-regular
elements, $h_T$ refers to the diameter of $T$, and $h_\cT\in L_\infty(\Omega)$ with
$h_\cT|_T:=h_T$ ($T\in\cT$) is the mesh-width function.
We also assume that $|J_T|$ is uniformly bounded so that $\|h_\cT\|_{L_\infty(\Omega)}=O(1)$
uniformly for all meshes $\cT$. This simplifies the writing of some norm estimates.

We recall the associated Piola transformation $\piola_T$ where, for $\btauref:\;\Tref\to \R^2$,
its transformed function $\btau:=\piola_T(\btauref)$ is defined through
\begin{align*}
  |J_T| \btau\circ F_T := \BB_T\btauref.
\end{align*}
Now, applying the Piola transform to tensor functions this does not maintain symmetry. We
therefore introduce a symmetrized version, the Piola--Kirchhoff transformation.
Specifically, for a tensor function $\Mref:\;\Tref \to \R^{d\times d}$, the transformed
function $\MM := \trafo_T(\Mref):\;T\to \R^{d\times d}$ is defined through
\begin{align*}
  |J_T| \MM\circ F_T := \BB_T\Mref\BB_T^t,
\end{align*}
cf.~\cite[Section~3.1]{PechsteinS_11_TDN}.
We collect some important transformation properties.

\begin{lemma}\label{lem:trafo}
  The transformation $\trafo_T:\;\HdDivref{\Tref}\to\HdDiv{T}$ is an isomorphism.
  Let $\Mref\in\HdDivref{\Tref}$, $\zref\in H^2(\Tref)$ and set $\MM:=\trafo_T(\Mref)$, $z:=\zref\circ
  F_T^{-1}\in H^2(T)$. Then, the relations
  \begin{align*}
    |J_T| (\div\Div\MM) \circ F_T &= \divref\Divref\Mref, \quad\text{and}\\
    |J_T| (\Grad\grad z \mathbin\colon \MM)\circ F_T &= \Gradref\gradref \zref \mathbin\colon\Mref
  \end{align*}
  hold so that, in particular,
  \begin{align}\label{eq:idboundary}
    \dual{\traceDD{T}\MM}{z}_{\partial T} = \dual{\traceDD{\Tref}\Mref}{\zref}_{\partial\Tref}
    \quad\forall\MM\in\HdDiv{T},\ \forall z\in H^2(T).
  \end{align}
  Furthermore,
  \begin{align}
    \label{eq:scalingDivDiv:a}
    \norm{\MM}{T} &\simeq h_T \norm{\Mref}{\Tref},
    \quad \norm{\div\Div\MM}{T} \simeq h_T^{-1}\norm{\divref\Divref\Mref}{\Tref}
  \end{align}
  for any $\Mref\in \HdDivref{\Tref}$ and with generic constants independent of $h$.
\end{lemma}

\begin{proof}
  The first two identities follow by definition of the transformation and tensor calculus.
  Relation~\eqref{eq:idboundary} follows by these properties and the definition \eqref{trGGT}
  of the trace operator,
  \begin{align*}
    \dual{\traceDD{T}\MM}{z}_{\partial T} &= \vdual{\div\Div\MM}{z}_T - \vdual{\MM}{\Grad\grad z}_T 
    = \vdual{\divref\Divref\Mref}{\zref}_{\Tref} - \vdual{\Mref}{\Gradref\gradref z}_{\Tref}
    \\&= \dual{\traceDD{\Tref}\Mref}{\zref}_{\partial\Tref}.
  \end{align*}
  The scaling properties~\eqref{eq:scalingDivDiv:a} follow by standard arguments.
\end{proof}

Combining the transformations $\trafo_T$ and $\piola_T$ we obtain a transformation for elements
of $\HDivdiv{T}$. The corresponding properties are obtained as before.

\begin{lemma} \label{la_trafo2}
  The transformation $(\trafo_T,\piola_T):\;\HDivdivref{\Tref}\to\HDivdiv{T}$ is an isomorphism. 
  Let $(\Xiref,\btauref)\in\HDivdivref{\Tref}$, $\uref\in H^2(\Tref)$ and set
  $(\XXi,\btau):=(\trafo_T(\Xiref),\piola_T(\btauref))$, $u:=\uref\circ F_T^{-1}\in H^2(T)$.
  The relations
  \begin{align}
    |J_T|(\Div\XXi-\btau) \circ F_T &= \Bmat_T(\Divref\Xiref-\btauref), \\
    |J_T| \grad u\cdot(\Div\XXi-\btau) \circ F_T &= \gradref\uref\cdot(\Divref\Xiref-\btauref)
  \end{align}
  hold, in particular
  \begin{align}
    \dual{\ttraceGG{T}u}{(\XXi,\btau)}_{\partial T}
    =
    \dual{\ttraceGG{\Tref}\uref}{(\Xiref,\btauref)}_{\partial\Tref}
    \quad\forall(\XXi,\btau)\in\HDivdiv{T},\ \forall u\in H^2(T).
  \end{align}
  Moreover,
  \begin{align}
    \norm{\Div\XXi-\btau}T \simeq \norm{\Divref\Xiref-\btauref}{\Tref} 
    \quad\text{and}\quad
    \norm{\div\btau}T \simeq h_T^{-1}\norm{\divref\btauref}{\Tref}.
  \end{align}
  for any $(\Xiref,\btauref)\in \HDivdivref{\Tref}$ and with generic constants independent of $h$.
\end{lemma}

\subsection{Basis functions for discrete trace spaces} \label{sec_basis}

Let us identify some basis functions for the trace space $Q_\cS$ (previously denoted by
$\widehat Q_\cS$) and introduce another piecewise polynomial trace space of $H^2(\cT)$.
We do this locally for an element $T\in\cT$.

Recall the space $U_{\trddiv{T}}$ defined in \eqref{UdDivT}.
For fixed $T\in\cT$, let $e_1,e_2,e_3$ be its nodes.
We consider the space $Q_T:=\traceDD{T}(U_{\trddiv{T}})$ and choose a basis
$\basq_j$, $j=1,\dots,9$, associated to its degrees of freedom~\eqref{eq:dofDDlocal} such that
\begin{align*}
  \dual{\basq_j}{z}_{\partial T} &=
  \left\{\begin{array}{ll}
    z(e_j) & j=1,2,3, \\
    \frac{1}{|E_{j-3}|}\int_{E_{j-3}} z\,ds & j=4,5,6, \\
    \int_{E_{j-6}} \nn_T\cdot \nabla z \,ds & j=7,8,9,
  \end{array}\right\}
  \qquad\forall z\in H^2(T).
\end{align*}
Here, $E_k$ denotes the edge spanned by $e_k$, $e_{\mathrm{mod}(k,3)+1}$.
Now select $\MM_j\in\HdDiv{T}$ with $\traceDD{T}\MM_j = \basq_j$.
We define $\basqref_j := \traceDD{\Tref}\trafo_T^{-1}\MM_j$.
Then, Lemma~\ref{lem:trafo} shows that
\begin{align}\label{eq:qj:traceId}
  \dual{\basq_j}{z}_{\partial T} = \dual{\basqref_j}{\zref}_{\partial\Tref}
  \quad\forall z\in H^2(T)\ \text{and}\ \zref=z\circ F_T^{-1}.
\end{align}
We also recall the local discrete space $\widehat U_T$ (now denoted by $U_T$), cf.~\eqref{hatUT}.
One notes that, if $u\in H^2(T)$ has the trace $\traceGG{T}(u)\in U_T$, this does not imply
that $\traceGG{\Tref}(\uref)\in U_{\Tref}$. In fact, affine maps of functions with polynomial traces
of degree three and polynomial normal derivatives of degree one can have polynomial normal derivatives
of degree two. For this reason we introduce the piecewise polynomial trace space
\begin{align*}
  P^{p,k}_c(\partial T)
  :=
  \traceGG{T}\{v\in H^2(T);\; v|_E \in P^p(E),\ \partial_{\nn_T} v|_E \in P^k(E)\
               \forall E\in\cE_T\}
\end{align*}
for $p\geq 3$, $k\geq 1$, and observe two things.
First, $U_{\partial T} = P^{3,1}_c(\partial T)\subset P^{3,2}_c(\partial T)$ and, second,
$u\in H^2(T)$ satisfies $\traceGG{\Tref}\uref \in P^{3,2}_c(\partial\Tref)$ if and only if
$\traceGG{T}u \in P^{3,2}_c(\partial T)$.
We also recall that $\dim(P^{3,1}_c(\partial T)) = 9$.
Furthermore, $\dim(P^{3,2}_c(\partial T)) = 12$.

\subsection{Fortin operator for the test space $H^2(\cT)$} \label{sec_projGG}

We start with the construction of a Fortin operator for the scalar test functions of $H^2(\cT)$.
We do this in two steps, starting with a preliminary operator $\auxGG$ and then taking care
of the kernel of $\pwGrad\pwgrad$.

\begin{lemma}\label{lem:auxGG}
  There exists an operator $\auxGG: H^2(\cT) \to P^3(\cT)$ such that
  \begin{align}\label{eq:idBouAuxGG}
    \dual{\basq}{\auxGG z}_\cS = \dual{\basq}{z}_\cS \quad\forall\basq\in Q_\cS, z\in H^2(\cT).
  \end{align}
  In particular, 
  \begin{align}\label{eq:idVolAuxGG}
    \vdual{\MM_h}{\pwGrad\pwgrad\auxGG z} = \vdual{\MM_h}{\pwGrad\pwgrad z} 
    \quad\forall\MM_h\in \PP^{0,s}(\cT).
  \end{align}
  Moreover, it holds the estimate
  \begin{align}\label{eq:boundAuxGG}
  \begin{split}
    \norm{\auxGG z}{} + \norm{h_\cT^2 \pwGrad\pwgrad \auxGG z}{}
    &\lesssim \norm{z}{} + \norm{h_\cT^2 \pwGrad\pwgrad z}{}
    \quad\forall z\in H^2(\cT).
  \end{split}
  \end{align}
\end{lemma}
\begin{proof}
  We construct the operator locally for each element $T\in\cT$. 
  The idea is to use a dual basis with $\dual{\basq_j}{\chi_k}_{\partial T} = \delta_{jk}$
  where $\chi_k\in P^3(T)$. Then, $\auxGG$ defined by
  \begin{align*}
    \auxGG z = \sum_{j=1}^9 \dual{\basq_j}{z}_{\partial T} \chi_j
  \end{align*}
  satisfies \eqref{eq:idBouAuxGG}.

  To construct the dual basis we consider a set of linearly independent functions. 
  Let $\eta_{1,2,3}$ denote the nodal functions, i.e,
  $\eta_j(e_k) = \delta_{jk}$ for $j,k=1,2,3$, and define
  \begin{align*}
    \eta_4 &:= \eta_1\eta_2,  \quad \eta_5 := \eta_2\eta_3, 
    \quad \eta_6 := \eta_3\eta_1, \quad\eta_\mathrm{b} := \eta_1\eta_2\eta_3 \\
    \eta_7 &:= \eta_4(\eta_2-\eta_1)+\eta_\mathrm{b}, 
    \quad\eta_8 := \eta_5(\eta_3-\eta_2)+\eta_\mathrm{b}, 
    \quad \eta_9 := \eta_6(\eta_1-\eta_3)+\eta_\mathrm{b}.
  \end{align*}
  Let $\Amat\in\R^{9\times 9}$ be the matrix with entries $\Amat_{jk} = \dual{\basq_j}{\eta_k}_{\partial T}$,
  $j,k=1,\dots,9$.
  If we prove that $\Amat$ is invertible, then the rows of $\Amat^{-1}$ define a set of linearly
  independent functions
  \begin{align*}
    \chi_k := \sum_{j=1}^9 \Amat^{-1}_{kj}\eta_j
  \end{align*}
  with the desired property
  \begin{align*}
    \dual{\basq_j}{\chi_k}_{\partial T} = \delta_{jk} \quad\text{for }j,k=1,\dots,9.
  \end{align*}
  Let us analyze the matrix $\Amat$. One verifies by simple calculations that $\Amat$ has the form
  \begin{align*}
    \Amat = \begin{pmatrix}
      \Amat_1 & \Amat_2 \\
      0 & \Amat_3
    \end{pmatrix},
  \end{align*}
  where $\Amat_1\in\R^{6\times 6}$ is an upper triangular matrix whose diagonal entries are non-zero.
  The matrix $\Amat$ is thus invertible if and only if the block $\Amat_3\in\R^{3\times 3}$ has full rank. 
  Again, direct calculations yield $\det(\Amat_3)\neq 0$. 
  At this point it is important to mention the appearance of the bubble function in the
  definition of $\eta_{7,8,9}$. Without adding the bubble function, the block $\Amat_3$ would not be invertible.
  Moreover, observe that, since all entries of $\Amat$ are independent of $h$, also the coefficients
  $\alpha_{jk}$ in $\chi_k = \sum_{j=1}^9 \alpha_{jk}\eta_j$ are independent of $h$ by~\eqref{eq:qj:traceId}.

  Next, we show identity~\eqref{eq:idVolAuxGG} element-wise. Let $\MM\in \PP^{0,s}(T)$.
  It holds $\traceDD{T}\MM \in Q_T$.
  The definition of the trace operator and~\eqref{eq:idBouAuxGG} yield
  \begin{align*}
    \vdual{\MM}{\Grad\grad \auxGG z}_T &= \vdual{\div\Div\MM}{\auxGG z}_T - \dual{\traceDD{T} \MM}{\auxGG z}_{\partial T}
    = -\dual{\traceDD{T} \MM}{\auxGG z}_{\partial T} \\
    &= -\dual{\traceDD{T} \MM}{z}_{\partial T} = \vdual{\div\Div\MM}{z}_T - \dual{\traceDD{T} \MM}{z}_{\partial T} \\
    &= \vdual{\MM}{\Grad\grad z}_T.
  \end{align*}
  It remains to prove estimate~\eqref{eq:boundAuxGG}. It follows by standard scaling arguments and
  an inverse inequality. Specifically, we prove the estimate for each $T\in\cT$. First, 
  \begin{align*}
    \norm{\auxGG z}{T}
    \leq \sum_{j=1}^9 |\dual{\basq_j}{z}_{\partial T}| \, \norm{\chi_j}T
    \lesssim \sum_{j=1}^9 h_T |\dual{\basqref_j}{\zref}_{\partial\Tref}|
    \lesssim h_T \norm{\zref}{2,\Tref}
    \simeq \norm{z}{T} + h_T^2\norm{\Grad\grad z}{T}.
  \end{align*}
  Second, applying an inverse inequality, we find that
  \begin{align*}
    \norm{\Grad\grad\auxGG z}T \lesssim h_T^{-2} \norm{\auxGG z}T 
    \lesssim h_T^{-2}\norm{z}T + \norm{\Grad\grad z}{T}.
  \end{align*}
  This finishes the proof.
\end{proof}

To finish the construction of the Fortin operator we
observe that $\ker(\pwGrad\pwgrad) = P^1(\cT)$.
For $z\in H^2(\cT)$, we denote by $z_{\ker}$ its $L_2(\Omega)$ projection onto $P^1(\cT)$.
Then we define
\begin{align*}
   \projGG:\;
   \left\{\begin{array}{cll}
      H^2(\cT) & \to &P^3(\cT),\\
      z        & \mapsto &\projGG z := \auxGG(z-z_{\ker})+z_{\ker},
   \end{array}\right.
\end{align*}
and show that it satisfies the required properties.

\begin{lemma} \label{la_projGG}
  For any $z\in H^2(\cT)$ the operator $\projGG:\;H^2(\cT) \to P^3(\cT)$ satisfies
  \begin{alignat}{2}
    \label{eq:idBouGG}
    \dual{\basq}{\projGG z}_\cS &= \dual{\basq}z_\cS &&\forall\basq\in Q_\cS,\\
    \label{eq:idVolGG}
    \vdual{\MM}{\pwGrad\pwgrad\projGG z} &= \vdual{\MM}{\pwGrad\pwgrad z}\quad
    &&\forall\MM\in \PP^{0,s}(\cT).
  \end{alignat}
  Furthermore,
  \begin{align}\label{eq:kerGG}
    \projGG z = z \quad\forall z\in \ker(\pwGrad\pwgrad)=P^1(\cT)
  \end{align}
  holds, and we have the local approximation property
  \begin{align}\label{eq:approxGG}
    \norm{z-\projGG z}{} \lesssim \norm{h_\cT^2 \pwGrad\pwgrad z}{}
    \quad\forall z\in H^2(\cT)
  \end{align}
  and the bound
  \begin{align*}
    \norm{\projGG z}{} + \norm{\pwGrad\pwgrad \projGG z}{}
    \lesssim
    \norm{z}{} + \norm{\pwGrad\pwgrad z}{}
    \quad\forall z\in H^2(\cT).
  \end{align*}
\end{lemma}
\begin{proof}
  Relation~\eqref{eq:kerGG} follows from the definition of $\projGG$.
  Then, identities~\eqref{eq:idBouGG},~\eqref{eq:idVolGG} follow from
  identities~\eqref{eq:idBouAuxGG},~\eqref{eq:idVolAuxGG}. Specifically,
  \begin{align*}
    \dual{\basq}{\projGG z}_\cS &= \dual{\basq}{\auxGG(z-z_{\ker})}_\cS 
    + \dual{\basq}{z_{\ker}}_\cS = \dual{\basq}{z-z_{\ker}}_\cS + \dual{\basq}{z_{\ker}}_\cS
    = \dual{\basq}{z}_\cS
  \end{align*}
  and, similarly, 
  \begin{align*}
    \vdual{\MM}{\pwGrad\pwgrad\projGG z} &= \vdual{\MM}{\pwGrad\pwgrad\auxGG (z-z_{\ker})} 
    + \vdual{\MM}{\pwGrad\pwgrad z_{\ker}} \\
    &= \vdual{\MM}{\pwGrad\pwgrad (z-z_{\ker})} + \vdual{\MM}{\pwGrad\pwgrad z_{\ker}} = 
    \vdual{\MM}{\pwGrad\pwgrad z}.
  \end{align*}
  Next, note that $z-\projGG z = z-\auxGG(z-z_{\ker})-z_{\ker}$, whence
  \begin{align*}
    \norm{z-\projGG z}{} \leq \norm{z-z_{\ker}}{} + \norm{\auxGG(z-z_{\ker})}{}.
  \end{align*}
  Then,~\eqref{eq:boundAuxGG} together with
  $\norm{z-z_{\ker}}{} \lesssim \norm{h_\cT^2 \pwGrad\pwgrad z}{}$ yield~\eqref{eq:approxGG}.

  Now, the approximation property shows that
  \begin{align*}
    \norm{\projGG z}{} \leq \norm{z-\projGG z}{} + \norm{z}{} \lesssim \norm{h_\cT^2 \pwGrad\pwgrad z}{} +
    \norm{z}{}.
  \end{align*}
  Finally, $\pwGrad\pwgrad\projGG z = \pwGrad\pwgrad\auxGG(z-z_{\ker})$ together
  with~\eqref{eq:boundAuxGG} and the approximation property yield
  \begin{align*}
    \norm{\pwGrad\pwgrad\projGG z}{}
    = \norm{\pwGrad\pwgrad\auxGG (z-z_{\ker})}{}
    \lesssim \norm{h_\cT^{-2}(z-z_{\ker})}{} + \norm{\pwGrad\pwgrad (z-z_{\ker})}{}
    \lesssim \norm{\pwGrad\pwgrad z}{}.
  \end{align*}
  This concludes the proof.
\end{proof}

\subsection{Fortin operator for the test space $\HDivdiv{\cT}$} \label{sec_projD}

Now we construct a Fortin operator for the combined tensor and vector valued test space
$\HDivdiv{\cT}$. It serves as the second component of the Fortin operator required for the analysis
of the stability of our fully discrete DPG scheme \eqref{DPG_discrete}. In fact, this
is the missing piece in the proof of Theorem~\ref{thm_DPG_discrete}.

By Lemma~\ref{la_trafo2} it holds $(\XXi,\btau)\in\HDivdiv{T}$ if and only if
$(\Xiref,\btauref) = \bigl(\trafo_T^{-1}(\XXi),\piola_T^{-1}(\btau)\bigr)\in \HDivdivref{\Tref}$
so that we start with the reference element $\Tref$.

\begin{lemma} \label{la_auxD}
  There exists
  \[
     \projDref:\;\left\{\begin{array}{cll}
        \HDivdivref{\Tref} &\to &\PP^{4,s}(\Tref) \times \bP^3(\Tref),\\
        (\Xiref,\btauref)  &\mapsto &
        \projDref(\Xiref,\btauref) = (\projDQref(\Xiref,\btauref),
                                      \projDtauref(\Xiref,\btauref))
     \end{array}\right.
  \]
  such that
  \begin{alignat}{2}
    \label{eq:auxComb:a}
    \dual{\uref}{\projDref(\Xiref,\btauref)}_{\partial\Tref} &=
    \dual{\uref}{(\Xiref,\btauref)}_{\partial\Tref}
    &&\forall\uref\in P^{3,2}_c(\partial\Tref),\\
    \label{eq:auxComb:b}
    \vdual{\Mref}{\projDQref(\Xiref,\btauref)}_{\Tref} &= \vdual{\Mref}{\Xiref}_{\Tref}
    &&\forall\Mref\in\PP^{0,s}(\Tref),\\
    \label{eq:auxComb:c}
    \vdual{\bthetaref}{\Div\projDQref(\Xiref,\btauref)-\projDtauref(\Xiref,\btauref)}_{\Tref} &=
    \vdual{\bthetaref}{\Div\Xiref-\btauref}_{\Tref}
    \quad&&\forall\bthetaref\in \bP^3(\Tref)
  \end{alignat}
  for any $(\Xiref,\btauref)\in\HDivdivref{\Tref}$.
  Moreover, the following bound holds true,
  \begin{align} \label{eq:auxComb:bound}
    \norm{\projDref(\Xiref,\btauref)}{\Divdivref,\Tref} \lesssim \norm{(\Xiref,\btauref)}{\Divdivref,\Tref}
    \quad\forall (\Xiref,\btauref)\in\HDivdivref{\Tref}.
  \end{align}
\end{lemma}
\begin{proof}
  We follow the procedure from Nagaraj \emph{et al.}~\cite{NagarajPD_17_CDF}.
  Specifically, we seek the element $\projDref(\Xiref,\btauref)$ with minimal norm
  \begin{align*}
    \projDref(\Xiref,\btauref)
    =
    \argmin\limits_{(\Xiref^*,\btauref^*)\in \PP^{4,s}(\Tref)\times\bP^3(\Tref)}
    \norm{(\Xiref^*,\btauref^*)}{\Divdivref,\Tref}
  \end{align*}
  subject to the conditions~\eqref{eq:auxComb:a}--\eqref{eq:auxComb:c}, that is,
  \begin{alignat*}{2}
    &\ip{(\Xiref^*,\btauref^*)}{(\delta\Xiref,\delta\btauref)}_{\HDivdivref{\Tref}} 
    \\ &\qquad\qquad+ \dual{\uref}{(\delta\Xiref,\delta\btauref)}_{\partial\Tref} 
    + \vdual{\Mref}{\delta\Xiref}_{\Tref}
    + \vdual{\bthetaref}{\Div\delta\Xiref-\delta\btauref}_{\Tref} &\,=\,& 0, \\
    & \dual{\delta\uref}{(\Xiref^*,\btauref^*)}_{\partial\Tref} &\,=\,&
    \dual{\delta\uref}{(\Xiref,\btauref)}_{\partial\Tref}, \\
    &  \vdual{\delta\Mref}{\Xiref^*}_{\Tref} &\,=\,& \vdual{\delta\Mref}{\Xiref}_{\Tref}, \\
    & \vdual{\delta\bthetaref}{\Div\Xiref^*-\btauref^*}_{\Tref} &\,=\,& 
      \vdual{\delta\bthetaref}{\Div\Xiref-\btauref}_{\Tref}
  \end{alignat*}
  for all $(\delta\Xiref,\delta\btauref) \in \PP^{4,s}(\Tref) \times \bP^3(\Tref)$, $\delta\uref \in
  P^{3,2}_c(\Tref)$, $\delta\Mref\in \PP^{0,s}(\Tref)$, $\delta\bthetaref\in \bP^3(\Tref)$.
  This is a linear system with matrix
  \begin{align*}
    \begin{pmatrix}
      \Amat & \Cmat \\
      \Cmat^\transp & 0
    \end{pmatrix}
  \end{align*}
  where $\Amat$ denotes the matrix associated to the inner product $\ip{\cdot}{\cdot}_{\HDivdivref{\Tref}}$.
  Clearly, the mixed formulation admits a unique solution if $\Cmat$ is injective.
  In our case this can be verified by some direct calculations (not shown).
  Since the solution $\projDref(\Xiref,\btauref) := (\Xiref^*,\btauref^*)$
  of the mixed problem depends continuously on the data, we conclude the boundedness
  estimate~\eqref{eq:auxComb:bound} with constants depending only on $\Tref$.
\end{proof}

We are now ready to conclude the existence of our second Fortin operator.

\begin{lemma} \label{la_projD}
  There exists
  \[
     \projD:\;\left\{\begin{array}{cll}
        \HDivdiv{\cT} &\to &\PP^{4,s}(\cT) \times \bP^3(\cT),\\
        (\XXi,\btau)  &\mapsto &
        \projD(\XXi,\btau) = (\projDQ(\XXi,\btau), \projDtau(\XXi,\btau))
     \end{array}\right.
  \]
  such that
  \begin{alignat}{2}
    \label{eq:comb:a}
    \dual{u}{\projD(\XXi,\btau)}_{\cS} &=
    \dual{u}{(\XXi,\btau)}_{\cS}
    &&\forall u\in U_\cS,\\
    \label{eq:comb:b}
    \vdual{\MM}{\projDQ(\XXi,\btau)} &= \vdual{\MM}{\XXi}
    &&\forall\MM\in\PP^{0,s}(\cT),\\
    \label{eq:comb:c}
    \vdual{\btheta}{\pwDiv\projDQ(\XXi,\btau)-\projDtau(\XXi,\btau)} &=
    \vdual{\btheta}{\pwDiv\XXi-\btau}
    \quad&&\forall\btheta\in\bP^3(\cT),\\
    \label{eq:comb:d}
    \vdual{z}{\pwdiv\projDtau(\XXi,\btau)} &=
    \vdual{z}{\pwdiv\btau}
    &&\forall z\in P^2(\cT)
  \end{alignat}
  for any $(\XXi,\btau)\in\HDivdiv{\cT}$.
  In particular, we have the commutativity properties
  \begin{align}
    \pwDiv\projDQ(\XXi,\btau)-\projDtau(\XXi,\btau)
    &= \Pi^3(\pwDiv\XXi-\btau), \label{eq:comb:com1}\\
    \pwdiv\projDtau(\XXi,\btau) &= \Pi^2 \pwdiv\btau. \label{eq:comb:com2}
  \end{align}
  Moreover, 
  \begin{align*}
    \norm{\projD(\XXi,\btau)}{\Divdiv,\cT} \lesssim \norm{(\XXi,\btau)}{\Divdiv,\cT}
    \quad\forall (\XXi,\btau)\in\HDivdiv{\cT}.
  \end{align*}
  Here, $\Pi^p:\;L_2(\Omega)\to P^p(\cT)$ denotes the $L_2(\Omega)$-projection.
\end{lemma}
\begin{proof}
  It suffices to consider one element $T\in\cT$.
  Let $(\XXi,\btau)\in \HDivdiv{T}$ with $\Xiref = \trafo_T^{-1}\XXi$, $\btauref = \piola_T^{-1}\btau$.
  We define $\projD(\XXi,\btau) := (\projDQ(\XXi,\btau),\projDtau(\XXi,\btau))$ with
  \begin{align*}
    \projDQ(\XXi,\btau) := \trafo_T(\projDQref(\Xiref,\btauref)),
    \text{ and }
    \projDtau(\XXi,\btau) := \piola_T(\projDtauref(\Xiref,\btauref)).
  \end{align*}
  To see the first identity~\eqref{eq:comb:a} we consider $u\in H^2(T)$ with
  $\traceGG{T}u\in P^{3,2}_c(\partial T)$ and $\uref = u\circ F_T$.
  The definition of $\traceDD{T}$, cf.~\eqref{trDDT}, and the properties of the
  transformations $\trafo_T,\piola_T$ yield
  \begin{align*}
    \dual{\traceDD{T}u}{(\XXi,\btau)}_{\partial T}
    &= \vdual{u}{\div\btau}_T + \vdual{\grad u}{\btau-\Div\XXi}_T - \vdual{\Grad\grad u}{\XXi}_T \\
    &= \vdual{\uref}{\divref\btauref}_{\Tref}
       + \vdual{\gradref\uref}{\btauref-\Divref\Xiref}_{\Tref} 
       - \vdual{\Gradref\gradref\uref}{\Xiref}_{\Tref} \\
    &= \dual{\traceDD{\Tref}\uref}{(\Xiref,\btauref)}_{\partial\Tref}.
  \end{align*}
  The same argumentation shows that
  \begin{align*}
    \dual{\traceDD{T}u}{\projD(\XXi,\btau)}_{\partial T} =
    \dual{\traceDD{\Tref}\uref}{\projDref(\Xiref,\btauref)}_{\partial\Tref}.
  \end{align*}
  Therefore,~\eqref{eq:comb:a} follows from~\eqref{eq:auxComb:a} and
  $U_{\partial T}=P^{3,1}_c(\partial T)\subset P^{3,2}_c(\partial T)$.
  
  To see~\eqref{eq:comb:b} and \eqref{eq:comb:c} we note that both transformations
  $\trafo_T$ and $\piola_T$ map polynomial tensor and vector functions, respectively,
  to polynomial functions of the same degree. Therefore, it is straightforward to show 
  that~\eqref{eq:comb:b} follows from~\eqref{eq:auxComb:b}. 
  Furthermore, observe that
  \begin{align*}
    |J_T| (\Div\XXi-\btau)\circ F_T = \Bmat_T(\Divref\Xiref-\btauref).
  \end{align*}
  Since $\Bmat_T$ is a matrix with constant coefficients,
  identity~\eqref{eq:comb:c} follows from~\eqref{eq:auxComb:c}.

  Note that the commutativity property~\eqref{eq:comb:com1} directly follows from~\eqref{eq:comb:c}.
  For the proof of the commutativity property~\eqref{eq:comb:com2} and
  identity~\eqref{eq:comb:d} let $z\in P^2(T)$.
  Note that $\ttraceGG{T}(z)=\traceGG{T}(z) \in P^{3,2}_c(\partial T)$,
  $\grad z\in \bP^1(T)$, $\Grad\grad z\in \PP^{0,s}(T)$.
  The definition of $\ttraceGG{T}$, cf.~\eqref{trGG2T}, and identities~\eqref{eq:comb:a}--\eqref{eq:comb:c} yield
  \begin{align*}
  \lefteqn{
    \vdual{z}{\div\projDtau(\XXi,\btau)}_T
  }\\
    &=
    \dual{\ttraceGG{T}z}{\projD(\XXi,\btau)}_{\partial T}
    -\vdual{\nabla z}{\projDtau(\XXi,\btau)-\Div\projDQ(\XXi,\btau)}_T
    +\vdual{\Grad\grad z}{\projDQ(\XXi,\btau)}_T \\
    &= \dual{\ttraceGG{T}z}{(\XXi,\btau)}_{\partial T}
    -\vdual{\nabla z}{\btau-\Div\XXi}_T + \vdual{\Grad\grad z}{\XXi}_T
    = \vdual{z}{\div\btau}_T.
  \end{align*}
  This shows~\eqref{eq:comb:d} and~\eqref{eq:comb:com2}.

  Finally, to show the boundedness we first consider the $L_2$-part of the norm.
  Scaling arguments and the boundedness of $\projDref$ prove
  \begin{align*}
    \norm{\projDQ(\XXi,\btau)}T \simeq h_T \norm{\projDQref(\Xiref,\btauref)}{\Tref}
    &\lesssim h_T(\norm{\Xiref}{\Tref} + \norm{\Divref\Xiref-\btauref}{\Tref} + \norm{\divref\btauref}{\Tref}) \\
    &\simeq \norm{\XXi}{T} + h_T \norm{\Div\XXi-\btau}{T} + h_T^2 \norm{\div\btau}{T}.
  \end{align*}
  To bound the other terms in the norm we employ the commutativity
  properties~\eqref{eq:comb:com1},\eqref{eq:comb:com2} to conclude that
  \begin{align*}
    \norm{\Div\projDQ(\XXi,\btau)-\projDtau(\XXi,\btau)}{T} + \norm{\div\projDtau(\XXi,\btau)}{T}
    &= \norm{\Pi^3(\Div\XXi-\btau)}T + \norm{\Pi^2\div\btau}T 
    \\ &\leq \norm{\Div\XXi-\btau}T + \norm{\div\btau}T.
  \end{align*}
  Combining the last two estimates finishes the proof.
\end{proof}

\begin{remark} \label{rem_Fortin_Divdiv}
We note that relations \eqref{eq:comb:c} and \eqref{eq:comb:d} have been used in
the proof of Theorem~\ref{thm_DPG_discrete} to verify the corresponding projection
properties \eqref{pf_DPG_discrete_2} and \eqref{pf_DPG_discrete_3}. There, piecewise
constant functions $\btheta_h\hateq\btheta\in\bP^0(\cT)$ resp. $u_h\hateq z\in P^0(\cT)$ are sufficient.
Piecewise polynomial functions $\btheta\in\bP^3(\cT)$ of higher degrees are needed in
\eqref{eq:comb:c} to prove the boundedness of the Fortin operator $\projD$ via
the commutativity properties \eqref{eq:comb:com1}, \eqref{eq:comb:com2}.
Furthermore, in this paper, relation~\eqref{eq:comb:d} for $z\in P^0(\cT)$ is sufficient.
\end{remark}

\subsection{Fortin operator for the test space $\HdDiv{\cT}$} \label{sec_projDD}

We now study a Fortin operator for test functions of $\HdDiv{\cT}$. As discussed
in \S\ref{sec_DPG_discrete2}, it guarantees well-posedness and quasi-optimal convergence
of the fully discrete scheme from \cite{FuehrerHN_UFK} when the polynomial degree
of the test space component for $\HdDiv{\cT}$ is increased from two to four.

Let us start with the construction on the reference element $\Tref$.

\begin{lemma}\label{lem:auxDD}
  There exists $\auxDD:\;\HdDivref{\Tref}\to \PP^{4,s}(\Tref)$ such that
  \begin{alignat}{2}
    \label{eq:auxDD:idbou}
    \dual{\uref}{\auxDD\Qref}_{\partial\Tref} &= \dual{\uref}{\Qref}_{\partial\Tref}
    \quad&&\forall \uref\in P^{3,2}_c(\partial \Tref),\\
    \label{eq:auxDD:idvol}
    \vdual{\Mref}{\auxDD\Qref}_{\Tref} &= \vdual{\Mref}{\Qref}_{\Tref}
    &&\forall \Mref\in\PP^{0,s}(\Tref)
  \end{alignat}
  for any $\Qref\in \HdDivref{\Tref}$.
  Moreover,
  \begin{align}\label{eq:auxDD:bound}
    \norm{\auxDD\Qref}{\dDivref,\Tref} \lesssim 
    \norm{\Qref}{\dDivref,\Tref} \quad\forall \Qref\in \HdDivref{\Tref}.
  \end{align}
\end{lemma}
\begin{proof}
  We follow the same idea as in Lemma~\ref{la_auxD} and define $\auxDD\Qref$ for given
  $\Qref\in\HdDivref{\Tref}$ as the solution $\Qref^*\in\PP^{4,s}(\Tref)$ of the mixed problem
  \begin{alignat*}{2}
    &\ip{\Qref^*}{\delta\Qref}_{\HdDivref{\Tref}} + \vdual{\Mref}{\delta\Qref}_{\Tref}
    + \dual{\uref}{\delta\Qref}_{\partial\Tref} &\,=\,& 0, \\
    &\vdual{\delta\Mref}{\Qref^*}_{\Tref} &\,=\,& \vdual{\delta\Mref}{\Qref}_{\Tref}, \\
    &\dual{\delta\uref}{\Qref^*}_{\partial\Tref} &\,=\,& 
    \dual{\delta\uref}{\Qref}_{\partial\Tref}
  \end{alignat*}
  for all $\delta\Qref\in\PP^{4,s}(\Tref)$, $\delta\Mref\in\PP^{0,s}(\Tref)$,
  $\delta\uref\in P^{3,2}_c(\partial \Tref)$.
  Here, $\ip{\cdot}{\cdot}_{\HdDivref{\Tref}}$ denotes the indicated inner product.
  The mixed problem is equivalent to the norm minimization problem subject to the
  constraints~\eqref{eq:auxDD:idbou},\eqref{eq:auxDD:idvol}.
  As in the proof of Lemma~\ref{la_auxD}, the mixed problem has a linear system with matrix of the form
  \begin{align*}
    \begin{pmatrix}
      \Amat & \Cmat \\
      \Cmat^\transp & 0
    \end{pmatrix}.
  \end{align*}
  Here, $\Amat$ corresponds to the inner product in $\HdDivref{\Tref}$.
  The mixed problem admits a unique solution if $\Cmat$ is injective.
  As before, this can be verified by direct calculations (not shown).

  Finally, boundedness follows with the same argumentation as in Lemma~\ref{la_auxD}.
\end{proof}

We are now ready to conclude the existence of our third Fortin operator.

\begin{lemma} \label{la_projDD}
  There exists $\projDD:\;\HdDiv{\cT}\to \PP^{4,s}(\cT)$ such that
  \begin{alignat}{2}
    \label{eq:projDD:idbou}
    \dual{u}{\projDD\QQ}_{\cS} &= \dual{u}{\QQ}_{\cS}
    &&\forall u\in U_{\cS}, \\
    \label{eq:projDD:idvol}
    \vdual{\MM}{\projDD\QQ} &= \vdual{\MM}{\QQ}
    &&\forall \MM\in\PP^{0,s}(\cT), \\
    \label{eq:projDD:idvol2}
    \vdual{z}{\pwdiv\pwDiv\projDD\QQ} &= \vdual{z}{\pwdiv\pwDiv\QQ}
    \quad&&\forall z\in P^2(\cT)
  \end{alignat}
  for any $\QQ\in\HdDiv{\cT}$.
  Moreover,
  \begin{align}
    \norm{\projDD\QQ}{\dDiv,\cT} &\lesssim \norm{\QQ}{\dDiv,\cT} \quad\text{and} \nonumber \\
    \pwdiv\pwDiv\projDD\QQ &= \Pi^2\pwdiv\pwDiv\QQ, \label{eq:projDD:com}
  \end{align}
  for any $\QQ\in\HdDiv{\cT}$.
  Here, as before, $\Pi^p:\;L_2(\Omega)\to P^p(\cT)$ denotes the $L_2(\Omega)$-projection.
\end{lemma}
\begin{proof}
  It suffices to define $\projDD$ element-wise and prove the identities and estimates element-wise. 
  Throughout let $T\in\cT$ be fixed. 
  For $\QQ\in\HdDiv{T}$ let $\Qref\in\HdDivref{\Tref}$ with $\QQ = \trafo_T(\Qref)$.
  Define $\projDD$ on $T$ by 
  \begin{align*}
    \projDD \QQ := \trafo_T(\auxDD\Qref) \in \PP^{4,s}(T).
  \end{align*}
  Let $u\in H^2(T)$ be such that $\traceGG{T}u\in P^{3,2}_c(\partial T)$, and let $\uref = u\circ F_T$.
  The definition of the trace operator and the properties of the transformation $\trafo_T$
  (Lemma~\ref{lem:trafo}) show that
  \begin{align*}
    \dual{\traceGG{T}u}{\projDD\QQ}_{\partial T} = \dual{\traceGG{\Tref}\uref}{\auxDD\Qref}_{\partial\Tref}
  \end{align*}
  The same argumentation yields
  \begin{align*}
    \dual{\traceGG{T}u}{\QQ}_{\partial T} = \dual{\traceGG{\Tref}\uref}{\Qref}_{\partial\Tref}.
  \end{align*}
  Therefore,~\eqref{eq:projDD:idbou} follows from~\eqref{eq:auxDD:idbou} and
  $U_{\partial T}\subset P^{3,2}_c(\partial T)$.
  For the second identity~\eqref{eq:projDD:idvol} note that for $\MM \in\PP^{0,s}(T)$ it holds
  $\Mref := \trafo_T^{-1}(\MM)\in\PP^{0,s}(T)$ as well.
  Then, it is easy to see that~\eqref{eq:projDD:idvol} follows from~\eqref{eq:auxDD:idvol}.
  For the proof of the last identity~\eqref{eq:projDD:idvol2} and the commutativity
  property~\eqref{eq:projDD:com} we use that for $z\in P^2(T)$ one has
  $\traceGG{T}z\in P^{3,2}_c(\partial T)$ and $\Grad\grad z\in \PP^{0,s}(T)$.
  The definition of $\traceGG{T}$, cf.~\eqref{trGGT}, and the identities established above show that
  \begin{align*}
    \vdual{z}{\div\Div\projDD\QQ}_T &= \vdual{\Grad\grad z}{\projDD\QQ}_T+\dual{\traceGG{T}z}{\projDD\QQ}_{\partial T}
    \\
    &= \vdual{\Grad\grad z}{\QQ}_T+\dual{\traceGG{T}z}{\QQ}_{\partial T} = \vdual{z}{\div\Div\QQ}_T.
  \end{align*}
  With the commutativity property~\eqref{eq:projDD:com} we see that
  \begin{align*}
    \norm{\div\Div\projDD\QQ}{T} = \norm{\Pi^2\div\Div\QQ}{T} \leq \norm{\div\Div\QQ}T.
  \end{align*}
  Thus, it only remains to bound the $L_2(T)$ part of the $\HdDiv{T}$-norm. 
  By scaling arguments (using the transformation $\trafo_T$) and by boundedness~\eqref{eq:auxDD:bound}
  we infer the relations
  \begin{align*}
    \norm{\projDD\QQ}{T}
    &\simeq h_T \norm{\auxDD\Qref}{\Tref}
    \lesssim h_T \norm{\Qref}{\Tref} + h_T \norm{\divref\Divref\Qref}{\Tref}
    \simeq \norm{\QQ}{T} + h_T^2 \norm{\div\Div\QQ}T.
  \end{align*}
  This finishes the proof.
\end{proof}

\begin{remark} \label{rem_Fortin_dDiv}
As in Remark~\ref{rem_Fortin_Divdiv}, we note that relation \eqref{eq:projDD:idvol2}
with $z\in P^0(\cT)$ is enough for the immediate orthogonality properties of the Fortin
operator $\projDD$, cf.~\eqref{orth_Fortin_dDiv_u} in~\S\ref{sec_DPG_discrete2}.
Higher polynomial degrees are required for the commutativity property \eqref{eq:projDD:com}
and the boundedness of $\projDD$.
\end{remark}

\section{A numerical example} \label{sec_num}

We present numerical experiments for a model problem with singular solution. It corresponds
to the second example in \cite{FuehrerHN_UFK}. The domain $\Omega$ and initial mesh $\cT$ are as indicated
in Figure~\ref{fig_mesh}.
The reentrant corner at $(x,y)=(0,0)$ has the interior angle $\tfrac3{4}\pi$.

\begin{figure}[htb]
\centerline{\includegraphics[width=0.4\textwidth]{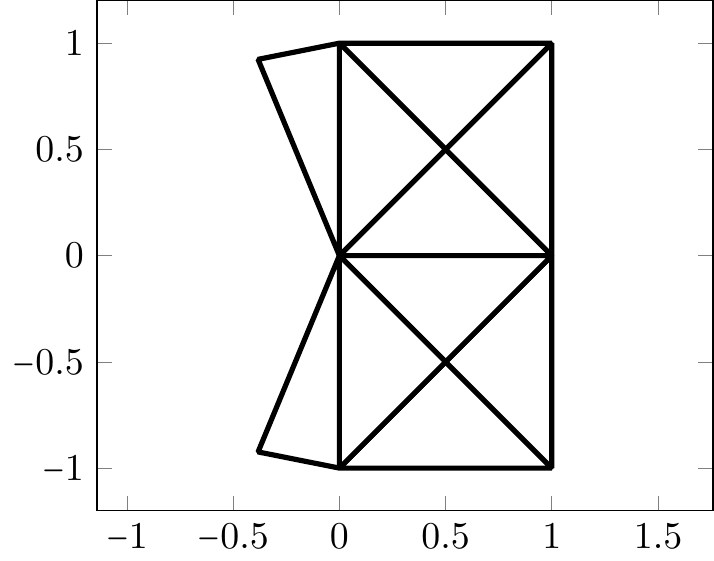}}
\caption{Geometry and initial mesh.}
\label{fig_mesh}
\end{figure}

We consider the manufactured solution
\begin{align*}
  u(r,\varphi) = r^{1+\alpha}(\cos( (\alpha+1)\varphi)+C \cos( (\alpha-1)\varphi))
\end{align*}
with polar coordinates $(r,\varphi)$ centered at the origin. It satisfies
\begin{align*}
  \div\Div\cC\Grad\grad u = 0 =: f
\end{align*}
with $\cC=I$ the identity,
and requires non-homogeneous boundary conditions specifying $u|_\Gamma$ and $\nabla u|_\Gamma$.
Selecting $f=0$ and $\cC=I$, we choose $\alpha$ and $C$ such that $u$ and its normal derivative
vanish on the edges that generate the incoming corner.
The approximate values are $\alpha\approx 0.67$ and $C\approx 1.23$.
It follows that $u\in H^{2+\alpha-\varepsilon}(\Omega)$,
$\btheta=\grad u\in H^{1+\alpha-\varepsilon}(\Omega)$,
$\MM = \Grad\grad u \in (H^{\alpha-\varepsilon}(\Omega))^{2\times 2}$ for $\varepsilon>0$.
Furthermore, $|\Div\MM(r,\varphi)|\simeq r^{\alpha-2}\not\in L_2(\Omega)$ so that
$\MM\in H(\div\Div,\Omega)$ and $\MM\notin \HH(\Div,\Omega)$ (the space of $\LL_2^s(\Omega)$-tensors
with divergence in $\bL_2(\Omega)$).

Now, for the numerical experiments, we take refinement steps with newest vertex bisection (NVB).
This generates families of shape-regular triangulations.
We either refine uniformly by dividing each triangle into four
sub-triangles of the same area, or we perform adaptive mesh refinement of the form
\begin{align*}
  \boxed{\texttt{SOLVE}} \quad\longrightarrow\quad
  \boxed{\texttt{ESTIMATE}} \quad\longrightarrow\quad
  \boxed{\texttt{MARK}} \quad\longrightarrow\quad
  \boxed{\texttt{REFINE}} \quad.
\end{align*}
As error indicators we use local (element) contributions $\eta(T)$ of the inherent (approximated) energy norm
of the DPG method,
\[
   \eta^2 :=
   \|B(\uu-\uu_h)\|_{\VV_h'}^2 = \sum_{T\in\cT} \|B(\uu-\uu_h)\|_{\VV_h(T)'}^2
   =: \sum_{T\in\cT} \eta(T)^2.
\]
Here, $\VV_h(T):=\VV_h|_T$, using that the discrete test space has a product structure associated
to the mesh. For an abstract analysis of this error estimator we refer to~\cite{CarstensenDG_14_PEC}.
For the marking step we use the bulk criterion
\begin{align*}
  \frac7{10} \eta^2 \leq \sum_{T\in\mathcal{M}} \eta(T)^2
\end{align*}
where $\mathcal{M}\subseteq\cT$ is the set of marked elements.
As usual we denote by $h:=h_\cT := \max_{T\in\cT} \diam(T)$ the discretization parameter.

\begin{figure}[htb]
  \centerline{\includegraphics[width=0.55\textwidth]{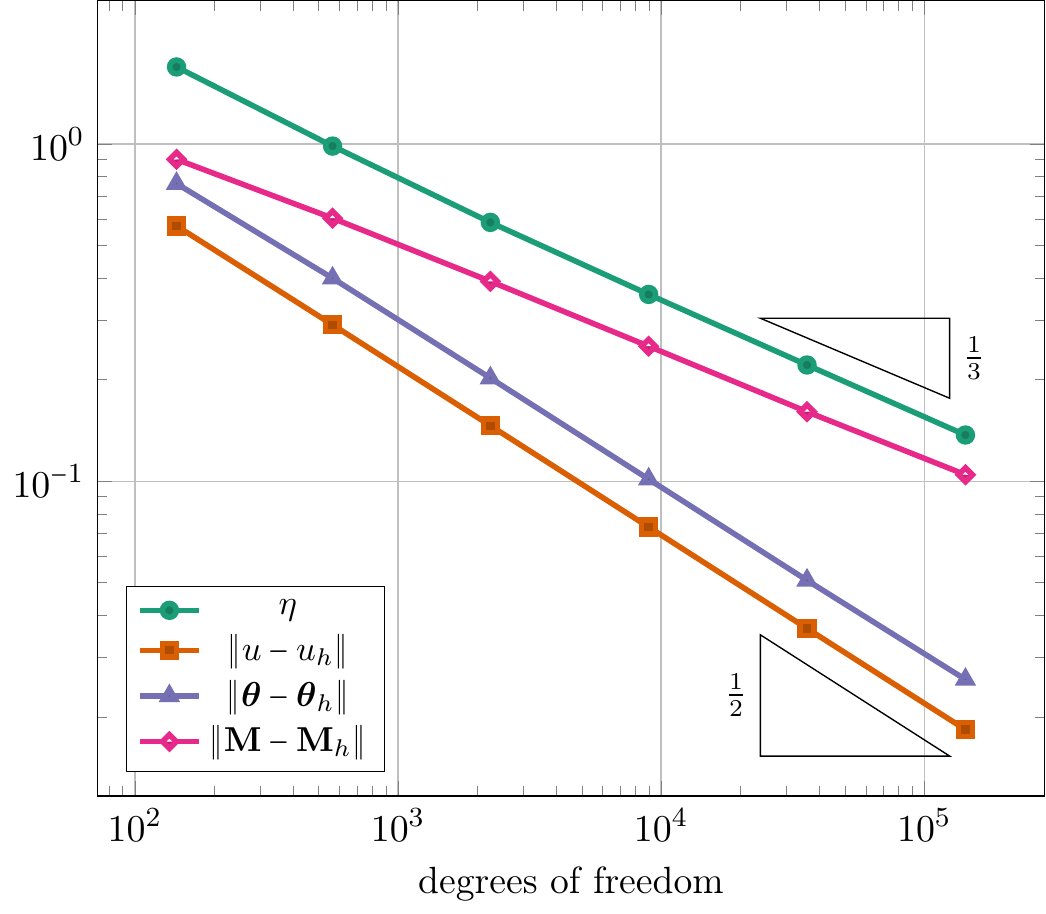}}
  \caption{Error curves for uniform mesh refinements.}
  \label{fig_unif}
\end{figure}

\begin{figure}[htb]
  \centerline{\includegraphics[width=0.55\textwidth]{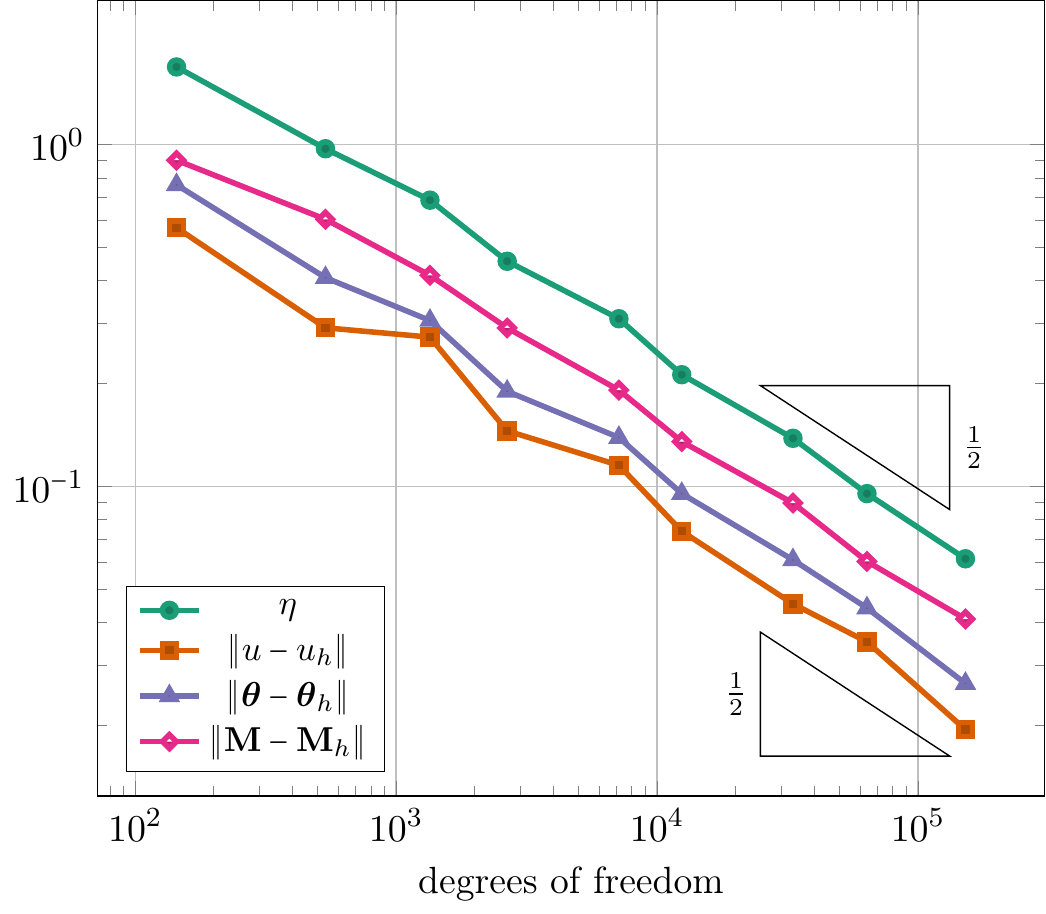}}
  \caption{Error curves for adaptively refined meshes.}
  \label{fig_adap}
\end{figure}

By the reduced regularity of $\MM$, we expect that uniform mesh refinements lead to a convergence order
$\OO(h^\alpha) = \OO(\dim(\UU_h)^{-\alpha/2})$ ($\alpha\approx 0.67$ by our selection).
Figure~\ref{fig_unif} confirms this rate. 
The numbers on the side of the triangles indicate negative slope with respect to $\dim(\UU_h)$.
The error curves for the energy norm $\eta$ and for $\MM$
in the $\LL_2^s(\Omega)$-norm both exhibit the expected slope, whereas the other field variables
$u$ and $\btheta=\grad u$ converge at a rate of about $O(h)$.
Figure~\ref{fig_adap} shows the corresponding results for adaptively refined meshes.
All field variables and the error in energy norm converge at a rate close to $O(h)$.

\bibliographystyle{siam}
\bibliography{/home/norbert/tex/bib/heuer,/home/norbert/tex/bib/bib}
\end{document}